\documentclass[12pt,times]{article}

\usepackage[margin=1in]{geometry}
\usepackage{graphicx,float}
\usepackage[title]{appendix}    
\usepackage[square, comma, numbers]{natbib}
\usepackage{setspace}
\usepackage{amsmath, amssymb, enumerate, amsthm,amsfonts,mathrsfs,mathtools,bbm}
\usepackage{booktabs}
\usepackage{enumitem}
\usepackage{xcolor}
\usepackage{algorithm, algpseudocode}
\usepackage[colorlinks=true,
            pdfpagemode=UseNone,
            urlcolor=blue,
            citecolor=blue,
            linkcolor=violet,
            bookmarks=true,
            backref=page
            ]{hyperref}
\newcommand{\andd}{\qquad\text{and}\qquad}
\newcommand{\ubsd}{u\in\bS^{d-1}}

\newcommand{\ind}[1]{\mathbbm{1}\left\{#1\right\}}

\newcommand{\rdd}{\mathbb{R}^{d}}
\newcommand{\re}{\mathbb{R}}

\newcommand{\Prr}[1]{\Pr\left(#1\right)}
\newcommand{\Prrr}[1]{\Pr\Big(#1\Big)}

\newcommand{\sam}[2]{\mathbb{#1}_{#2}}

\newcommand{\norm}[1]{\lVert#1\rVert}

\newcommand{\eqd}{\stackrel{\text{d}}{=}}

\newcommand{\am}[2]{\widetilde{\cM}_1(#1,#2)}
\DeclareMathOperator{\SM}{SM}
\DeclareMathOperator{\BP}{BP}

\DeclareMathOperator{\TV}{TV}
\DeclareMathOperator{\interior}{int}

\newcommand{\E}[2]{\mathbb{E}_{#1}#2}

\DeclareMathOperator*{\lap}{ {\rm Laplace}}
\DeclareMathOperator*{\argmin}{ {\rm argmin}}

\DeclarePairedDelimiter\floor{\lfloor}{\rfloor}

\DeclareMathOperator{\med}{MED}

\DeclareMathOperator{\mad}{MAD}
\DeclareMathOperator{\tad}{TAD}
\DeclareMathOperator{\tm}{TM}
\DeclareMathOperator{\EM}{EM}

\newcommand{\cM}{\mathcal{M}}

\newcommand{\cN}{\mathcal{N}}

\newcommand{\cC}{\mathcal{C}}
\newcommand{\cP}{\mathcal{P}}
\newcommand{\cQ}{\mathcal{Q}}

\newcommand{\sB}{\mathscr{B}}
\newcommand{\cU}{\mathcal{U}}

\newcommand{\bX}{{\mathbb{X}}}

\newcommand{\bD}{\mathbf{D}}

\newcommand{\bfS}{\mathbf{S}}

\newcommand{\bS}{\mathbb{S}}

\newcommand{\hmu}{{\hat{\nu}}}

\DeclareMathOperator{\vol}{vol}

\DeclareMathOperator{\KL}{KL}

\newtheorem{theorem}{Theorem}
\newtheorem{corollary}[theorem]{Corollary}

\newtheorem{lemma}[theorem]{Lemma}

\theoremstyle{remark}
\newtheorem{remark}[theorem]{Remark}

\theoremstyle{definition}

\newtheorem{definition}{Definition}
\newtheorem{example}{Example}
\newtheorem{condition}{Condition}

\numberwithin{equation}{section}
\numberwithin{theorem}{section}
\numberwithin{example}{section}
\numberwithin{definition}{section}

\title{Differentially private projection-depth-based medians}
\author{Kelly Ramsay \& Dylan Spicker}
\begin{document}
\maketitle
\begin{abstract}
We develop $(\epsilon,\delta)$-differentially private projection-depth-based medians using the propose-test-release (PTR) and exponential mechanisms. Under general conditions on the input parameters and the population measure, (e.g. we do not assume any moment bounds), we quantify the probability the test in PTR fails, as well as the cost of privacy via finite sample deviation bounds. 
Next, we show that when some observations are contaminated, the private projection-depth-based median does not break down, provided its input location and scale estimators do not break down. 
We demonstrate our main results on the canonical projection-depth-based median, as well as on projection-depth-based medians derived from trimmed estimators. In the Gaussian setting, we show that the resulting deviation bound matches the known lower bound for private Gaussian mean estimation. In the Cauchy setting, we show that the ``outlier error amplification'' effect resulting from the heavy tails outweighs the cost of privacy. This result is then verified via numerical simulations. Additionally, we present results on general PTR mechanisms and a uniform concentration result on the projected spacings of order statistics, which may be of general interest.
\end{abstract}
\section{Introduction}
Privacy has become a central concern in data collection and analysis owing to considerations around personal autonomy, freedom of expression, and fairness and equality. 
Differential privacy has emerged as a powerful framework for maintaining privacy in data collection and analysis. 
It has recently been adopted by many high profile institutions, including the United States Census Bureau \citep{Abowd20222020}. 
This has spurred investigations of many classical statistical problems through a differentially private lens. 
Among such problems is median estimation.

Early literature on differentially private median estimation focuses on the univariate setting \citep{Dwork2009, Avella-Medina2019a, Brunel2020, zamos2020}. 
In the multivariate setting, several authors have studied differentially private medians through via differentially private mean estimation. 
These works have mainly focused either on the Tukey median \citep{Brown2021,2021Liub,2022arXiv220807438B}, or variants of it \citep{Hopkins2021}. 
In addition, as part of a more general framework for private high-dimensional estimation, \citet{2021Liu} studied one instance of a projection-depth-based median based on the trimmed mean and the trimmed variance. 
Given that their focus was mean estimation, all of these works enforce a bound on the number of moments possessed by the population measure. 
However, a classical method for assessing the robustness of a median is to examine its behavior in the case where the population measure possesses no moments. 
Recently, \citet{ramsay2022concentration} considered this problem. 
They studied differentially private multivariate medians arising from depth functions, including the Tukey median, without enforcing any moment bounds on the population measure. 
However, their framework does not cover projection-depth-based medians \citep{Zuo2003, Zuo2004}. 




Unlike the Tukey median and those studied by \citet{ramsay2022concentration}, when the location and scale measures are chosen appropriately, projection-depth-based medians simultaneously possess affine equivariance and achieve a high breakdown point \citep{Zuo2003, Zuo2004}. 
Affine equivariance is a critical property for location estimators, as it ensures that the units of measurement do not have an effect on the resulting estimate, see, for example, \citet{Zuo2000}. 
In addition, the breakdown point is a common measure of the robustness of an estimator. It measures the fraction of points which must be corrupted to make an estimate arbitrarily large \citep{Hampel2005}. 
For instance, the breakdown point of the Tukey median lies between $(1+d)^{-1}$ and $1/3$ \citep{Liu_2017}, whereas the projection-depth-based medians generally retain a much higher finite sample breakdown point, (e.g., approximately $1/2-d/n$ \citep{Zuo2003}). 
This finite sample breakdown point is carried over to the private setting, (see Theorem~\ref{thm::bdp}). 
One may also note that, when the location and scale measures are chosen appropriately, the projection-depth-based medians enjoy higher relative efficiency than that of the Tukey median \citep{Zuo2003}. 
In addition, \cite{Depersin2022} recently considered the behavior of a new subclass of projection-depth-based medians inspired by the median of means estimator. 
They demonstrate under general assumptions, e.g., the population measure need not have any moments, and under adversarial contamination this subclass of medians achieves sub-Gaussian rates with respect to Mahalanobis error. 
Motivated by these properties, our work develops and studies differentially private projection-depth-based medians. 
We consider a general location estimation setting and do not place any moment bounds on the population measure. 

Our main tool for privatizing projection-depth-based medians is the propose-test-release (PTR) algorithm, which was introduced by \citet{Dwork2009} in the context of differentially private estimation of various robust, univariate statistics. 
Our work takes inspiration from that work, as well as from the works of \citet{Brunel2020} and \citet{Brown2021}. 
\citet{Brunel2020} used PTR to develop sub-Gaussian median estimators in the univariate setting, and
\citet{Brown2021} developed a private Gaussian mean estimator, using a PTR mechanism based on safety and the exponential mechanism. We incorporate both ideas into our proposed mechanism.   
Notably, the mechanism proposed by \citet{Brown2021} was generalized by \citet{2021Liu} to be an all-purpose solution to various private, statistical problems. 
Our work builds on that of \citet{2021Liu}, in the sense that their mean estimator fits into the class of projection-depth-based estimators. 

Our contributions are as follows: We first propose a class of $(\epsilon,\delta)$-differentially private projection-depth-based medians. 
Next, under general conditions on the population measure, we quantify both the probability that the PTR mechanism fails to release an estimate and the cost of privatizing the projection-depth-based medians (Theorems~\ref{thm::no-reply}, and~\ref{thm::pd_acc2}, respectively). 
To measure the cost of privacy, we provide finite sample bounds on the distance between a private projection-depth-based median and the corresponding population value, which hold with high probability. 
Next, we show that when some observations are contaminated, the private projection-depth-based median does not break down, provided its input location and scale estimators do not break down (Theorem~\ref{thm::bdp}). 
We then demonstrate our methods for the canonical projection-depth-based median, as well as on projection-depth-based medians derived from trimmed estimators. 
In particular, we show that Theorems~\ref{thm::no-reply} and~\ref{thm::pd_acc2} apply provided the ``projected densities'' are bounded below in an interval containing a range of inner quantiles (Condition~\ref{cond::dens_lb}). 
As a first example, we show that in the Gaussian setting, our deviation bound matches the optimal lower bound for sub-Gaussian mean estimation given by \citet{Kamath2018} and \citet{Cai2019} (Example~\ref{ex::gaus}). 
We then show that in the Cauchy setting, the increase in error resulting from contamination or heavy tails, which we call the ``outlier error amplification'' is larger than the cost of privacy (Example~\ref{ex::cauchy}). 
We confirm this phenomenon with numerical studies (Section~\ref{sec::sim}). 
Lastly, we provide some results for analysis of general PTR mechanisms (Section~\ref{sec::ptr_general}), as well as a uniform concentration result on the projected spacings of order statistics (Lemma~\ref{thm::order_spacing_new_uniform}), which may be of general interest. 

\section{Preliminaries}
\subsection{Outlyingness and projection-depth-based medians}
We first review projection-depth-based medians, starting with \emph{projected outlyingness} \cite{Stahel1981,Donoho1982,Donoho1992, Zuo2000, Zuo2003}. 
Let $\mathcal{M}_1(\re^d)$ be the set of measures on $\mathscr{B}(\rdd)$, where $\mathscr{B}(S)$ denotes the Borel sets of $S$. 
For a measure $\nu\in\cM_1(\rdd)$ and $u$ from the unit $(d-1)$-sphere, $\bS^{d-1}$, let $\nu_u$ be the law of $X^\top u$ if $X\sim \nu$. 
\begin{definition}
Given a translation and scale equivariant estimator $\mu$ and a translation invariant and scale equivariant estimator $\sigma$, the projected outlyingness at $x\in\rdd$ with respect to $\nu$ is
\begin{equation}\label{eqn::outlyingness}
O(x,\nu)=\sup _{u\in \bS^{d-1}}O_u(x,\nu_u)=\sup _{u\in \bS^{d-1}} \frac{\left|x^\top u-\mu(\nu_u)\right|}{\sigma(\nu_u)}. 
\end{equation}
\end{definition}
For brevity, we will refer to the projected outlyingness as simply the outlyingness, and denote $\mu(\nu_u)$ as $\mu_u$, $\sigma(\nu_u)$ as $\sigma_u$, $\mu(\hat{\nu}_u)$ as $\hat\mu_u$, and $\sigma(\hat{\nu}_u)$ as $\hat\sigma_u$. 
The sample version of the outlyingness function is given by $O(x,\hat{\nu})$. 
The canonical choices of $\mu$ and $\sigma$ are the median and the median absolute deviation, respectively. 
The outlyingness based on the median and the median absolute deviation was first developed by \citet{Stahel1981} and \citet{Donoho1982}. 
It was later studied in general by \citet{Zuo2003}. 
Choosing $\mu$ and $\sigma$ to be the trimmed mean and trimmed variance recovers the objective function used in the PTR mechanism developed by \citet{2021Liu} for differentially private mean estimation. 
Lastly, choosing $\mu$ and $\sigma$ to be the median of means and a median of means version of the mad recovers the outlyingness studied by \citet{Depersin2022}. 
In our examples, we consider the trimmed mean, the median, trimmed mean absolute deviation, median absolute deviation, and the interquartile range. 
The minimizer of the outlyingness function is a \emph{projection-depth-based median} \citep{Stahel1981,Donoho1982,Zuo2003}. 
\begin{definition}
Given a translation and scale equivariant estimator $\mu$ and a translation invariant and scale equivariant estimator $\sigma$, the projection-depth-based median of $\nu$ is
$$\theta\coloneqq \theta(\mu,\sigma,\nu)=\int_{\argmin O} x dP,$$ where $P$ is uniform on $\argmin_{x\in\rdd} O(x,\nu)$.  
\end{definition}

The projection-depth-based medians enjoy several nice properties, including local and global robustness, affine equivariance, strong $\sqrt{n}$-consistency, and a relatively simple limiting distribution \citep{Zuo2003, Zuo2004}. 
To the best of our knowledge, among the competing affine equivariant medians, when the location and scale measures are chosen appropriately, (e.g. the median and median absolute deviation), the projection-depth-based medians have the highest breakdown point \citep{Zuo2003}. 
In Section \ref{sec::PD}, we show that when some observations are contaminated, the private projection-depth-based median does not break down, provided its input location and scale estimators do not break down (Theorem~\ref{thm::bdp}).  
Projection-depth-based medians have been shown to have higher asymptotic efficiency, relative to competitors \citep{Zuo2003}. 
In addition, recently, \citet{Depersin2022} showed that the median of means version of the projection-depth-based median exhibits sub-Gaussian rates, under adversarial contamination and fairly general assumptions. 
For more information on projection-depth-based medians, we refer the reader to \citep{Stahel1981,Donoho1982,Zuo2003, Zuo2004,Depersin2022}. 
\subsection{Differential privacy}\label{sec::DP_intro}
Next, we review necessary concepts from differential privacy. 
For a textbook introduction to differential privacy, see \citet{Dwork2014}.
A \emph{dataset} of size $n\times d$ is a collection of $n$ points in $\rdd$, $\mathbb{X}_n=\{X_i\}_{i=1}^n$, possibly with repetitions.
Let $\mathbf{D}_{n\times d}$ be the set of all datasets of size $n\times d$. 
For datasets $\mathbb{X}_n$ and $\mathbb{Y}_n$, let $\mathbb{X}_n \triangle \mathbb{Y}_n$ represent the dataset containing all rows in $\mathbb{X}_n$ which are not contained in $\mathbb{Y}_n$. 
Define $\bfS(\mathbb{X}_n,m)=\left \{ \mathbb{Y}_n\in \mathbf{D}_{n\times d} \colon  |\mathbb{X}_n \triangle \mathbb{Y}_n|=m \right\}$, the set of datasets of size $n\times d$ which differ from $\mathbb{X}_n$ in exactly $m$ rows. 
Define the empirical measure of $\mathbb{X}_n$, $\hmu_{\mathbb{X}_n}$, as that which assigns $1/n$ probability to each row in $\mathbb{X}_n$.
For a dataset with empirical measure $\hmu_{\mathbb{X}_n}$ define
$$\widetilde{\cM}_1(\hmu_{\mathbb{X}_n},m)=\{\tilde{\nu}\in\cM_1(\rdd)\colon \tilde{\nu}=\hmu_{\mathbb{Y}_n} \text{ for some } \mathbb{Y}_n\in \bfS(\mathbb{X}_n,m) \}.$$

Suppose that we observe a dataset $\bX_n$ comprised of $n$ i.i.d.\ samples from a measure $\nu\in\cM_1(\rdd)$, and wish to estimate a parameter $\theta\in\Theta$. 
One way to do this is via a \emph{mechanism}.  
Formally, a mechanism, $P_{\hat\nu}$, maps an empirical measure $\hat\nu$ to $\cM_1(\Theta)$. 
The estimator $\tilde\theta$ is then taken to be a draw from $P_{\hat\nu}$. 
When the corresponding measure is clear from context, we will suppress it in the notation, denoting $P_{\hat\nu}$ as $P$.
To produce a differentially private statistic, $P_{\hat\nu}$ should satisfy the following privacy guarantee.
\begin{definition}\label{def::dp}
A mechanism $P$ is $(\epsilon,\delta)$-differentially private if for any dataset $\mathbb{X}_n$, any $\tilde{\nu}\in \widetilde{\cM}_1(\nu_{\mathbb{X}_n},1)$, and any measurable set $B\in\mathscr{B}(\Theta)$, it holds that
\begin{equation}\label{eqn::dp}
    P_{\hat{\nu}_{\mathbb{X}_n}}(B)\leq e^\epsilon P_{\tilde{\nu}}(B)+\delta.
\end{equation} 
\end{definition}
The pair $(\epsilon,\delta)$ represent the privacy budget, where typically $\epsilon$ is a small constant and $\delta=O(n^{-k})$ for some positive integer $k$. Smaller values of both $\epsilon$ and $\delta$ enforce stricter privacy guarantees. 
The parameter $\epsilon$ controls the level of privacy assured to each individual, and $\delta$ is the probability of leakage, or, the probability that the privacy condition $P_{\hat{\nu}_{\mathbb{X}_n}}(B)\leq e^\epsilon P_{\tilde{\nu}}(B)$ fails to hold. 
If $\delta = 0$, the mechanism is said to be $\epsilon$-differentially private.
\section{Main results}\label{sec::PD}
\subsection{Private projection-depth-based medians}
In this section, we introduce private projection-depth-based medians. 
Private projection-depth-based medians rely on an instantiation of the \emph{exponential mechanism} \citep{McSherry2007}. 
For some $\tau>0$, let $A_{\tau,\hat{\nu}}=\{O(x ,\hat{\nu})\leq \tau\}$. 
Given $\tau,\eta,\epsilon>0$, the exponential mechanism applied to the outlyingness function $\EM_{\hat{\nu}}(\cdot;\tau,\eta,\epsilon)$ is the measure whose density satisfies
\begin{equation*}
d\EM_{\hat{\nu}}(\cdot;\tau,\eta,\epsilon)\propto\ind{x\in A_{\tau,\hat{\nu}}}\exp(-O(x,{\hat{\nu})\epsilon/4\eta})dx.
\end{equation*}
For brevity, we write $\EM_{\hat{\nu}}(\cdot;\tau,\eta,\epsilon)$ as simply $\EM_{\hat{\nu}}$ when $\tau,\eta,\epsilon$ are clear from the context. 
Now, this mechanism alone is not differentially private. 
That is because there are some datasets which are not ``safe''. 
\emph{Safety}, introduced by \citet{Brown2021} for Gaussian mean estimation, is a localized version of the $(\epsilon,\delta)$-differential privacy guarantee. 
We present a definition specific to our context. (For a generalized definition, see Section~\ref{sec::ptr_general}.) 
Take $\widehat\cM_{1}(n,d)$ to be the set of empirical measures constructed from datasets of size $n\times d$, $\{ \nu_{\sam{X}{n}}\colon\ \sam{X}{n}\in  \mathbf{D}_{n\times d}\}$. 
\begin{definition}\label{dfn::sfty}
Given $\epsilon,\delta,\tau,\eta>0$, $\hat\nu$ is $(\epsilon,\delta)$-safe if, for all $\tilde\nu\in\widetilde\cM(\hat\nu,1)$ and $B\in\sB(\rdd)$, it holds that
$$\EM_{\hat{\nu}}(B)\leq e^{\epsilon}\EM_{\tilde{\nu}}(B)+\delta\qquad\text{and}\qquad \EM_{\tilde{\nu}}(B)\leq e^{\epsilon}\EM_{\hat{\nu}}(B)+\delta.$$
The $(\epsilon,\delta)$-safe set is defined as
$\bfS(\epsilon,\delta)= \{ \hat\nu\in \widehat{M}_1(n,d)\colon\  \hat\nu \text{ is } (\epsilon,\delta)\text{-safe} \}.$
\end{definition}
Indeed, for all common choices of $\mu,\ \sigma$, such as the median and median absolute deviation, or the trimmed mean and variance, there exists $\hat\nu\notin\bfS(\epsilon,\delta)$, which implies that $\EM_{\hat{\nu}}$ is not differentially private. 
However, because the common choices of $\mu$ and $\sigma$ are robust statistics, in the sense that they are not corrupted under small levels of contamination, we expect most $\hat{\nu}$ to be safe. 
When we expect $\hat{\nu}$ to be safe with high probability, it is natural to develop a \emph{propose-test-release} (PTR) mechanism \citep{Dwork2009}. 

PTR mechanisms are characterized by a three-step procedure. 
First, you \textbf{propose} a mechanism $P_{\hat\nu}$. 
Second, you privately \textbf{test} if $\hat\nu$ is safe. 
Finally, if the test is passed, you \textbf{release} a draw from $P_{\hat\nu}$. If the test is not passed, nothing is released. 
In the event that nothing is released, the PTR mechanism still uses up the $(\epsilon,\delta)$ privacy budget. 
As such, it is critical that the test is passed with high probability. 
For our estimators, we take $P_{\hat\nu}=\EM_{\hat{\nu}}$. 
Since the test step of the algorithm is differentially private, it is a mechanism, and so we refer to it as the test mechanism. 
Our proposed mechanism's test mechanism relies on the \emph{safety margin}.  
\begin{definition}\label{dfn::sfty_mar}
Given $\epsilon,\delta>0$, the safety margin of $\hat\nu$ is 
$$\SM(\hat\nu)\coloneqq \SM(\hat\nu; \epsilon,\delta)=\inf\{m\in \mathbb{Z}^+\colon\ \exists \tilde\nu\in\tilde{\cM}(\hat\nu,m)\ s.t.\ \tilde\nu\notin \bfS(\epsilon,\delta) \}.$$
\end{definition}
The safety margin is the minimum number of points that we need to modify in the original dataset in order for it to become unsafe. 
Observe that if $\SM(\hat\nu)>0$, then $\hat\nu$ is safe. 
Let $W$ be a standard Laplace random variable. 
The test is defined as follows: 
\begin{equation}\label{eqn::test_m}    
Z=\ind{\SM(\hat\nu)+\frac{2}{\epsilon}W>\frac{2\log(1/2\delta)}{\epsilon}},
\end{equation}
where, if $Z=1$, then the test is passed. 

We can now introduce the differentially private projection-depth-based median which, given $\mu$ and $\sigma$, outputs a private version of $\theta$, the projection-depth-based median based on $\mu$ and $\sigma$.  
%
\begin{definition}
Given parameters $\epsilon,\delta,\eta,\tau>0$ and estimators $\mu$ and $\sigma$, 
the private projection-depth-based median is defined as 
$$\tilde{\theta}=\begin{cases}
    Y\sim \EM_{\hat\nu}& Z=1\\
    \perp & Z=0
\end{cases} .$$
\end{definition}
Note that if $\tilde{\theta}=\perp$, it is meant that nothing is output by the algorithm. 
That is, $\perp$ can be interpreted as a null value. 
The private projection-depth-based median is $(\epsilon,\delta)$-differentially private (Theorem~\ref{thm::s_suff}). 
Note that the mechanisms of \citet{Brown2021} and \citet{2021Liu} are PTR mechanisms based on $Z$ and the exponential mechanism. 
\citet{Brown2021} considers a different cost function and \citet{2021Liu} consider several cost functions, one of which is the outlyingness based on the trimmed mean and variance. 
The performance of the proposed PTR algorithm is naturally summarized by both the probability that the test is passed ($\E{}{Z}$) and the accuracy of the approximation of $\theta$ by $\tilde{\theta}$. 
With these two quantities, we characterize both the likelihood that the mechanism releases a statistic, and how useful that statistic is likely to be. 

To analyze these quantities, we first introduce the \emph{projected sensitivity} and the \emph{maximal absolute projected deviation} of an arbitrary, univariate estimator\footnote{By univariate estimator, we mean a functional on the space $\cM_1(\re)$.}, $T$, respectively defined as $$\tilde{S}_{n,k}(T)=\sup_{\substack{\tilde{\nu}\in \am{\hat\nu}{k}\\ u\in \bS^{d-1}}}|T(\hat\nu_{u})-T(\tilde{\nu}_u)| \qquad\text{and}\qquad S_{n}(T)=\sup_{u\in \bS^{d-1}}|T(\hat\nu_{u})-T(\nu_u)|.$$
The first quantity ($\tilde{S}_{n,k}(T)$), which we call the projected sensitivity, is an extension of sensitivity. It is the maximum change in $T$ when $k$ observations are set to arbitrary values, over all directions. 
The second quantity ($S_{n}(T)$), referred to as the maximal absolute projected deviation of $T$, is the maximum distance between the sample value of $T$ and the population value of $T$, over all directions. 
Using these definitions, we introduce three regularity conditions required for our main results. For $a,b\in\re$, we write $a\lesssim b$ ($a\gtrsim b$) if there exists a universal constant $c>0$ such that $a\leq cb$ ($a\geq cb$). 

\begin{condition}\label{cond::LS_Bound}
For all $d\geq 1$ and $k\leq 3$, there exists constants $a_{1}>0$ and $a_{2}>0$, and a sequence $\kappa_{n}\longrightarrow 0$ as $n\rightarrow \infty$ such that for all $n\geq 1$, $T\in\{\mu,\sigma\}$, and $t\geq 0$
$$\Prr{\tilde{S}_{n,k}(T)\geq t}\leq a_{1}\exp\left(-a_{2}nt\right)+\kappa_{n}.$$
\end{condition}
\begin{condition}\label{cond::Pop_Bound}
For all $d,n\geq 1$ and $T\in\{\mu,\sigma\}$ there exists a sequence of non-increasing functions $\zeta_{n,T}\longrightarrow 0$ as $n\to \infty$ such that for all $t\geq 0$ 
$$\Prrr{S_{n}(T)\geq t}\lesssim \zeta_{n,T}(t).$$
\end{condition}
\begin{condition}\label{cond::bounded-parameters}
There exists constants $c_1,c_2,c_3>0$ such that, for the population measure $\nu$, $0<c_1\leq \inf_{\ubsd} \sigma_u\leq \sup_{\ubsd}\sigma_u\leq c_2<\infty$ and $\sup_{\ubsd} |\theta^\top u- \mu_u|\leq c_3<\infty$.
\end{condition}

Condition~\ref{cond::LS_Bound} can be interpreted as stating that ``$\mu$ and $\sigma$ are robust to $k/n$-contamination with high probability.'' 
To elaborate, Condition~\ref{cond::LS_Bound} says that the probability that one can contaminate $k$ observations in the sample so that $\hat\mu_u$ or $\hat\sigma_u$ moves by at least $t$ in at least one direction $u$ is small. 
This condition may seem difficult to check, and so as a result, we provide examples verifying it in Section~\ref{sec::examples}. 
Condition~\ref{cond::Pop_Bound} requires that $\hat\mu_u$ and $\hat\sigma_u$ concentrate around their respective population values, uniformly in $u$.
Typically, this can be shown using standard empirical process or sub-Gaussian concentration techniques, (see for example, Lemma~\ref{lem::cond_2_med} and Lemma~\ref{lemm::cond2-gaussian-trimmed}). 
Condition~\ref{cond::bounded-parameters} is standard for projection-depth-based estimators (e.g., \citep{Zuo2003, Zuo2004}). 
Under these three conditions, we present a lower bound on the probability that the proposed PTR mechanism passes the test.
\begin{theorem}\label{thm::no-reply}
Suppose that Conditions~\ref{cond::LS_Bound}--\ref{cond::bounded-parameters} hold. Then there exists a universal constant $C>0$ such that for all $d,n\geq 1$, $\delta\leq 1/2$, $\epsilon,\tau>0$ and $\eta<1$ such that $\tau>4\eta+16c_3/c_1$ and
$$\delta> \exp\left(-\frac{\epsilon \tau}{8\eta}+\frac{\epsilon}{2}+d\log\left(\frac{2c_2 (\tau+2\eta)+4c_3}{c_1\tau/4-\eta-4c_3}\right) \right),$$
it holds that 
\begin{equation*}
    \E{}{Z}\geq 1-C\left(e^{\epsilon}\delta+a_{1}\exp\left(-\frac{a_{2}c_2n \eta}{4(\tau\vee 1)}\right)+\kappa_{n}+\zeta_{n,\sigma}(c_1/4)\right).
\end{equation*}
\end{theorem}
The proof of Theorem~\ref{thm::no-reply} is given in Appendix~\ref{app:proofs}. 
To interpret Theorem~\ref{thm::no-reply} note that, given robust $\mu$ and $\sigma$ in the sense of Condition~\ref{cond::LS_Bound}, the probability of failing the test in the mechanism is governed by  the ratio $a_{2}\eta/\tau$ and $\delta$ (recalling that $e^\epsilon\approx 1$).
Next, we show that the mechanism will be accurate with high probability. 
The following theorem bounds the deviation of the private projection-depth-based median around its population value. 
\begin{theorem}\label{thm::pd_acc2}
Suppose Conditions~\ref{cond::Pop_Bound} and~\ref{cond::bounded-parameters} hold. 
Then, for all $n,d\geq 1$, $\epsilon,\eta>0$, $\delta\leq 1/2$, $\gamma~\in~(3(\zeta_{n,\sigma}(c_1/2)\vee \zeta_{n,\mu}(c_3))\vee 3(1-\E{}{Z}),1]$ and $\tau\geq 20(2c_3/(10c_2 - c_1)+c_3/c_1)$ with probability $1-\gamma$, it holds that 
\begin{equation*}
    \norm{\tilde{\theta}-\theta}\lesssim \frac{ c_2\eta}{\epsilon}\left[\log(1/\gamma)\vee d\log\left(\frac{2c_2\tau}{c_3}+1 \right)\vee\frac{c_3}{c_1}\right] \vee \frac{c_2}{c_1}\zeta_{n,\mu}^{-1}(\gamma/3) .
\end{equation*}
\end{theorem}
The proof of Theorem~\ref{thm::pd_acc2} is given in Appendix~\ref{app:proofs}. 
Theorem~\ref{thm::pd_acc2} shows that, provided $\eta$ is relatively small and $\E{}{Z}$ is large enough, the cost of privacy for privately estimating the projection-depth-based median is governed by $\eta(\log(1/\gamma)\vee d)/\epsilon$. 
Together, Theorem~\ref{thm::no-reply} and Theorem~\ref{thm::pd_acc2} quantify the trade-off inherent in $\eta$. 
Smaller $\eta$ improves accuracy at the cost of a higher probability of failing the test. 
In general, it is desirable for the cost of privacy to be less than the sampling error, $c_2\zeta_{n,\mu}^{-1}(\gamma/3)/c_1$. 
We find this to be the case in the examples explored in the next section for robust choices of $\mu$ and $\sigma$. 
Indeed, for such choices, in the following examples, we have that $\kappa_n$ and $\zeta_{n,\sigma}(c_1/2)\vee \zeta_{n,\mu}(c_3)$ are exponentially small in $n$, as well as $\eta$ can be taken to be $O(c_d/n)$, where $c_d$ is a term which is at most polynomial in the dimension, $d$. 
This results in a failure probability $\lesssim \delta e^{\epsilon}$ and a cost of privacy which is less than the sampling error.

\begin{remark}
It is popular to consider the error of an estimator with respect to the Mahalanobis metric, e.g. \citep{Brown2021,Depersin2022}, as it accounts for the differing scales across different directions. 
However, we aim to develop a general theory which includes population measures for which the Mahalanobis metric may not be suitable. 
As a result, we develop results based on the Euclidean metric. 
However, \citet{2021Liu} and \citet{Depersin2022} have considered the behavior of different subclasses of projection-depth-based medians (both privately and non-privately, respectively). 
Their results show good performance of these projection-depth-based medians, i.e., under their assumptions, they achieve sub-Gaussian rates. 
Furthermore, if we restrict to the median of means version of the projection-depth-based median and the assumptions of \citet{Depersin2022}, one can straightforwardly substitute the techniques of \citet{Depersin2022} into the proof of Theorem \ref{thm::pd_acc2} where appropriate, which would give the rate of convergence of the median of means version of the private projection-depth-based median under the Mahalanobis metric. 
\end{remark}

Before concluding this section, we present a result concerning the robustness of the private projection-depth-based median. 
The finite sample breakdown point of an estimator is the maximum fraction of points which can be set to arbitrary values, under which the estimator lies in the interior of the parameter space \citep{Donoho1983}. 
For a set $A$, let $\interior(A)$ denote the interior of $A$. 
Formally, for an estimator $T$ of parameter $\theta\in\Theta$, the finite sample breakdown point of $T$ with respect to $\hat\nu$ is
$$\BP(T,\hat\nu)=\sup\{\frac{m}{n}\colon \forall \tilde\nu\in\cM_1(\hat\nu,m),\ T(\tilde\nu)\in \interior(\Theta)\}.$$
Given $\hat\nu\in\widehat\cM_{1}(n,d)$, define
\begin{multline*}
\BP(\EM,\hat\nu)=\sup\{\frac{m}{n}\colon \forall \tilde\nu\in\cM_1(\hat\nu,m),\exists\ C_{\hat\nu,\alpha,\tau}<\infty \\
\text{such that}\EM_{\tilde\nu}(\{\norm{X}<C_{\hat\nu,\alpha,\tau}\})=1\}.
\end{multline*}
$\BP(\EM,\hat\nu)$ is the maximum fraction of points which can be set to arbitrary values, under which the estimate is bounded with probability 1. (Note that if $\norm{\tilde\theta}<\infty$, then $\tilde\theta\in\interior(\Theta)$.)  
We now present the following result, pertaining to the robustness of the private projection-depth-based median. 
\begin{theorem}\label{thm::bdp}
For all $0<\eta,\delta,\epsilon,\tau<\infty$ it holds that $\BP(\EM,\hat\nu)\geq \BP(\mu,\hat\nu)\wedge \BP(\sigma,\hat\nu).$
\end{theorem}
\begin{proof}
Note that by definition $\EM_{\hat\nu}(\{x\colon O(x,\hat\nu)\leq\tau\})=1$. 
Therefore, it suffices to show that the set $\{x\colon O(x,\tilde\nu)\leq\tau\}$ is bounded for all $\tilde\nu\in\cM_1(\hat\nu,k)$, where $k<n(\BP(\mu)\wedge \BP(\sigma))$. 
Consider $\tilde\nu\in\cM_1(\hat\nu,k)$, where $k<n(\BP(\mu)\wedge \BP(\sigma))$. 
We have that $$\{x\colon O(x,\tilde\nu)\leq\tau\}\subset \{x\colon \tau \inf_u\sigma(\tilde\nu_u)-\sup_u\mu(\tilde\nu_u)\leq x^\top u\leq \tau \sup_u\sigma(\tilde\nu_u)+\sup_u\mu(\tilde\nu_u)\}.$$ This set is bounded, provided that $\mu(\tilde\nu_u)$ and $\sigma(\tilde\nu_u)$ are bounded and $\sigma(\tilde\nu_u)>0$. Furthermore, $\mu(\tilde\nu_u)$ and $\sigma(\tilde\nu_u)$ are bounded and $\sigma(\tilde\nu_u)>0$ follows from the fact that $\tau<\infty$, the condition on $k$ and the definition of the breakdown point. 
\end{proof}
Therefore, if the data is contaminated such that $\mu$ and $\sigma$ do not break down, then Theorem \ref{thm::bdp} implies that the differentially private projection-depth-based median will also not breakdown.

\subsection{Verifying Condition~\ref{cond::LS_Bound} and example applications}\label{sec::examples}
We next verify Condition~\ref{cond::LS_Bound} for several choices of $\mu$ and $\sigma$, before applying Theorem~\ref{thm::no-reply} and Theorem~\ref{thm::pd_acc2} for Gaussian and Cauchy population measures. 
We first introduce some new notation. 
For $\nu\in \cM_1(\rdd)$, let $F_{\nu}$ be the cumulative distribution function and $f_{\nu}$ be the Radon--Nikodym derivative of $\nu$. 
If $d=1$, let $F^{-1}_{\nu}$ be the left-continuous quantile function and let $\xi_{q,u}=F^{-1}_{\nu_u}(q)$. 
We now define the univariate estimators considered in this section. 
First, we have the trimmed mean and trimmed mean absolute deviation:
\begin{definition}
Denote the $i$th order statistic of a sample, $\{X_1,\dots,X_n\}$, as $X_{(i)}$. Then, for $0<\alpha<1/2$ and $\sam{X}{n}\in\bD_{n\times 1}$, the $\alpha$-trimmed mean and $\alpha$-trimmed mean absolute deviation is given by:
    \begin{equation*}
       \tm_\alpha(\hat\nu_{\sam{X}{n}})= \sum_{i=\floor{n\alpha}}^{n-\floor{n\alpha}}\frac{X_{(i)}}{n(1-2\alpha)}\andd\tad_\alpha(\hat\nu_{\sam{X}{n}})= \sum_{i=\floor{n\alpha}}^{n-\floor{n\alpha}}\frac{|X_{(i)}- \tm_\alpha(\hat\nu_{\sam{X}{n}})|}{n(1-2\alpha)}.
    \end{equation*}
\end{definition}
We could also consider the trimmed standard deviation, as in \cite{2021Liu}, but the trimmed mean absolute deviation is more convenient mathematically. 
We next define the canonical estimators used in the projected outlyingness, the median and median absolute deviation. 
For a dataset $\sam{X}{n}\in\bD_{n\times 1}$ and a real function $g$, define $g(\sam{X}{n})=\{g(X_1),\ldots,g(X_n)\}$. 
\begin{definition}
Given $\sam{X}{n}\in\bD_{n\times 1}$, the median is given by $\med(\hat\nu_{\sam{X}{n}})=\tm_{1/2-1/2n}(\hat\nu_{\sam{X}{n}})$ and the median absolute deviation is given by $\mad(\hat\nu_{\sam{X}{n}})=\tm_{1/2-1/2n}(\hat\nu_{|\sam{X}{n}-\med(\hat\nu_{\sam{X}{n}})|})$.
\end{definition}
For these estimators, Condition~\ref{cond::LS_Bound} reduces to the following condition on the pair $(q,\nu)$, for a certain $0<q\leq 1/2$. 
\begin{condition}\label{cond::dens_lb}
The measure $\nu$ is absolutely continuous and there exists $r,M>0$ such that, for the interval $\mathcal{I} = [\xi_{q-3/n,u}-r, \xi_{1-q+3/n,u}+r]$, we have $\inf_{\substack{x\in\mathcal{I}; \ubsd}}f_{\nu_u}(x)>M$.
\end{condition}
For a given quantile, Condition~\ref{cond::dens_lb} requires that the density of the projected population measure is bounded below in some interval which contains $(\xi_{q-3/n,u}, \xi_{1-q+3/n,u})$. 
For $q=1/2$, Condition~\ref{cond::dens_lb} is a multivariate extension of the conditions imposed on the univariate PTR median \citep{Brunel2020} and the univariate median of \citet{zamos2020}. 
Under Condition~\ref{cond::dens_lb}, the trimmed estimators, as well as the median and median absolute deviation satisfy Condition~\ref{cond::LS_Bound}.
\begin{theorem}\label{thm::trimmed}
For all $0<\alpha<1/2$ if Condition~\ref{cond::dens_lb} holds for $(\alpha,\nu)$ then there exists universal constants $c,c'>0$ such that the trimmed mean and trimmed absolute deviation satisfy Condition~\ref{cond::LS_Bound} with $a_{1}=\exp({d+\log(n(1-2\alpha))})$, $a_{2}=cMn(1-2\alpha)$ and $\kappa_{n}=\exp({-nr^2M^2+d\log (c'n/d)}).$
\end{theorem}
\begin{theorem}\label{thm::order_cond_1}
Suppose there exists $0<q'<q<1/4$ such that both $6<n(q-q')$ and Condition~\ref{cond::dens_lb} holds for $(q',\nu)$. 
Then the median absolute deviation satisfies Condition~\ref{cond::LS_Bound} with $a_{1}\lesssim n\exp(d)$, $a_{2}=3M/24$ and $\kappa_{n}=\exp({-nr^2M^2+d\log (cn/d)}).$
\end{theorem}
\noindent The proofs are deferred to Appendix~\ref{app::ex_pro}, and follow mainly from Lemma~\ref{thm::order_spacing_new_uniform}. 
Theorem~\ref{thm::trimmed} and Theorem~\ref{thm::order_cond_1} can be interpreted as stating that, provided the (univariate) projected densities are bounded below when computed at the inner quantiles, the trimmed estimators and the median absolute deviation have small projected sensitivity with high probability. 
Note that we can prove an analogue of Theorem~\ref{thm::order_cond_1} with the interquartile range in place of the median absolute deviation (see Lemma~\ref{lem::iqr_cond_1}). 
\noindent To illustrate the methodology, we begin with a simple example, taking $\nu$ to be Gaussian.  
\begin{example}[Gaussian data]\label{ex::gaus}
Take the population measure to be $\cN(\theta,\Sigma)$ with $\theta\in\rdd$ and $\Sigma$ as a positive definite symmetric matrix. 
Without loss of generality, we may assume that $\nu = \cN(0,\Lambda)$, where $\Lambda$ is a diagonal matrix with eigenvalues $\beta_1 \leq \cdots \leq \beta_d$ and let $\omega$ be the condition number of $\Lambda$, i.e., $\omega=\beta_d/\beta_1$.
Specifically, the bounds produced by Theorems~\ref{thm::no-reply} and~\ref{thm::pd_acc2} when $\nu = \cN(0,\Lambda)$ are equivalent to those produced by taking $\nu=\cN(\theta,\Sigma)$, where $\Sigma=R\Lambda R^\top$ for some unitary matrix $R$. 
To see this, note that $R\nu + \theta=\cN(\theta,\Sigma)$, where $A\nu+b$ represents the law of $AX+b$ if $X\sim \nu$. 
Given that the outlyingness function is invariant under affine transformations, with $O(x,\nu) = O(Ax + b, A\nu + b)$, and using the fact that the private median is computed using only the outlyingness function, the simplified case loses no generality. 

First, we check Conditions~\ref{cond::Pop_Bound}--\ref{cond::dens_lb} for $(\mu,\sigma)=(\tm_\alpha,\tad_\alpha)$ where $1-2\alpha=O(1)$. 
Direct calculation yields that Condition~\ref{cond::bounded-parameters} is satisfied with $c_1\leq \beta_1^{1/2}$, $c_2\geq \beta_d^{1/2}$ and $c_3>0$. Take $c_3=1$.  
Next, we check Condition~\ref{cond::dens_lb} for $q=\alpha$ and $r=\beta_1^{1/2}/4$.
Provided that $\delta>e^{-n\epsilon/4\alpha}/2$, then Condition~\ref{cond::dens_lb} holds with $$M=(2\pi\beta_d)^{-1/2}\exp({-(\Phi^{-1}(\alpha-3/ n)-1/4)^2/2})\gtrsim \beta_d^{-1/2}.$$ 
We now apply Theorem~\ref{thm::trimmed} to show that Condition~\ref{cond::LS_Bound} holds. 
Doing so yields that for some universal constant $c'>0$, we have that $a_{1}\lesssim \exp({d+\log n})$,  $a_{2}\gtrsim \beta_d^{-1/2}n$, and $\kappa_n\lesssim \exp\left(-c'n/\omega+d\log(cd/n)\right).$
Next, using Gaussian concentration of measure, (Lemma~\ref{lemm::cond2-gaussian-trimmed}) Condition~\ref{cond::Pop_Bound} is satisfied for $(\tm_\alpha,\tad_\alpha)$ with 
$$\zeta_{n,\mu}(t)=\zeta_{n,\sigma}(t)= \exp\left( -cnt^2/\beta_d+d\right),$$
for some universal constant $c>0$. 

We can now apply Theorems~\ref{thm::no-reply} and~\ref{thm::pd_acc2}. 
Beginning with Theorem~\ref{thm::no-reply}, there exists universal constants $C,C'>0$, for
$$\delta\gtrsim \exp\left(-\epsilon\left(\frac{ \tau}{\eta}\wedge n\right)+d\log(C'\omega+C)\right),$$
there exists universal constants $c,c',c''>0$ such that
\begin{equation*}
          \E{}{Z}
      \geq  1-c''\left(\delta e^\epsilon+ \exp\left(-c'\frac{n^2\eta }{\beta_d^{1/2}(\tau\vee 1)}+d+\log n\right)
      +\exp\left(-\frac{c'n}{\omega}+d\log(cn/d)\right)\right).
\end{equation*}

Thus, in order to control the probability of failing the test, it suffices to inspect the ratio 
$n^2\eta/\tau.$
Note that $O$ is a standardized measure of outlyingness, and doesn't depend on the scale of the data, therefore it is natural to take $\tau$ to be a constant. 
Next, suppose $\delta=O(n^{-L})$ for some large $L>0$. 
Taking $\eta\gtrsim 1/n\beta_d^{1/2}$ results in $\E{}{Z}\geq 1-c\delta e^{\epsilon}$, provided that $\beta_d^{1/2}=O(1)$. 

Next, we can take $\eta\propto 1/n$ and apply Theorem~\ref{thm::pd_acc2}, which yields that, 
provided $n\geq d$, for all $c e^{\epsilon} n^{-L}\lesssim \gamma \leq 1,$ and $d\geq 1$, with probability $1-\gamma$,
\begin{equation*}
    \norm{\tilde{\theta}-\theta}\lesssim \frac{\log(1/\gamma)\vee d\log(\beta_d+1)}{n\epsilon} \vee \sqrt{\omega\frac{\log(1/\gamma)\vee d\log(n/d)}{n}}.
\end{equation*}
For many reasonable values of $\gamma$, the sampling error greatly exceeds the cost of privacy. 
For instance, for $\gamma \propto \delta$, gives $||\tilde{\theta}-\theta||\approx \sqrt{\omega\cdot d\log(n/d)/n}$ whenever, roughly, $n\gtrsim d/\omega\epsilon^2$. 
Indeed, whenever $n\gtrsim d/\omega\epsilon^2$, the deviations bound obtained from Theorem~\ref{thm::pd_acc2} matches the optimal deviations bound (for sub-Gaussian mean estimation) given by \citet[][see Theorem 3.1]{Cai2019}, see also \citep{Kamath2018}, up to logarithmic terms. 
\end{example}
In our next example, we quantify the cost of privacy in the heavy-tailed setting, i.e., when no moments of $\nu$ exist. 
Given that the mean does not exist, we consider $(\mu,\sigma)=(\med,\mad)$. 
We show that even in this difficult setting, the cost of privacy is small relative to the sampling error. 
Furthermore, we observe that the ``outlier error amplification'' resulting from the heavy tails, greatly outweighs the cost of privacy. 
\begin{example}[Cauchy data]\label{ex::cauchy}
Consider $\nu$ to be made up of $d$ independent Cauchy marginals with scale parameter equal to one and location parameter equal to zero. Analogous to Example~\ref{ex::gaus}, the assumptions of the scale parameter being one and location parameter being zero are made without loss of generality.
Again, we check Conditions~\ref{cond::Pop_Bound}--\ref{cond::dens_lb}. 
First, direct calculation yields that Condition~\ref{cond::bounded-parameters} holds with $c_1=1$, $c_2=d^{1/2}$ and $c_3=1$. 

Provided that $\delta\geq e^{-cn\epsilon }$ for some universal constant $c>0$, Condition~\ref{cond::dens_lb} holds with $M\gtrsim d^{-1/2}$. 
Applying Theorem~\ref{thm::order_cond_1}, we conclude that Condition~\ref{cond::LS_Bound} holds with $a_{1}\lesssim \exp(d+\log n)$, 
$a_{2,d,3}\gtrsim d^{-1/2}$, and $\kappa_n\lesssim \exp(-c'n/d+d\log\left(cn/d\right))$. 
Next, we use Condition~\ref{cond::dens_lb} to show that Condition~\ref{cond::Pop_Bound} holds (see Lemma~\ref{lem::cond_2_med} and Lemma~\ref{lem::cond_2_mad}) with 
$$\zeta_{n,\mu}(t)= \exp\left(-2nt^2/d+d\log(cn/d)\right),$$
and
$$\zeta_{n,\sigma}(t)= \exp\left(-c'nt^2/d+d\log(cn/d)\right)+
\exp\left(-c'n/d^2+d\log(cn/d)\right),$$
for universal constants $c,c'>0$.
We have now shown that all necessary conditions for the application of Theorems~\ref{thm::no-reply} and~\ref{thm::pd_acc2} are satisfied. 
We begin with an application of Theorem~\ref{thm::no-reply} and omit universal constants for clarity. 
For $\delta\gtrsim \exp\left(-\epsilon \tau/\eta+\epsilon+d\log(d\tau)\right)$, we have
\begin{align*}
     \E{}{Z}&\geq  1-(e^{\epsilon}\delta-\exp\left(-\frac{n\eta }{d^{1/2}\tau}+d+\log n\right)- \exp\left(-\frac{n}{d^2}+d\log(n/d)\right).
\end{align*}
Thus, compared to the Gaussian case, the lower bound on $\delta$ is larger. 
Furthermore, we require $\eta\gtrsim d^{3/2}/n\vee d^{1/2}\log n/n$ ensuring that $\E{}{Z}\gtrsim  n^{-L}$, a factor of $d^{1/2}$ larger than in the Gaussian case. 
We can now apply Theorem~\ref{thm::pd_acc2}. 
Taking $\eta\propto d^{3/2}/n $, for $n^{-L}<\gamma<1$ with probability $1-\gamma$, it holds that
$$
\norm{\tilde{\theta}-\theta} \lesssim\frac{d^{3/2}\left[\log (1 / \gamma)\vee d\log d\right]}{n\epsilon} \vee  \sqrt{\frac{d\log(1/\gamma)\vee d^{2}\log(n/d)}{n}}.
$$
It is immediate that the scaling in the dimension of the deviations bound is worse than that of the Gaussian case (Example~\ref{ex::gaus}). 
This is expected, given that heavy-tailed estimation is a more challenging problem. 
We refer to this as ``outlier error amplification''. 
A similar phenomenon, in the form of sample complexity bounds, was observed by \citet{ramsay2022concentration}, for private medians based on other depth functions.  
However, it is important to note that the obtained bound is only an upper bound on the deviations of the estimator. 
Indeed, to the best of our knowledge, the lower bound on the sample complexity of differentially private estimation of the location parameter from a product of Cauchy marginals is unknown. 
It is straightforward to show that this is at least as hard as differentially private Gaussian mean estimation, see Lemma~\ref{lem::s_c_cauchy}. 
However, there is a gap between the upper bounds derived here and in \citet{ramsay2022concentration} and the lower bound for Gaussian mean estimation. 
We conjecture that a tighter lower bound exists, which matches the upper bound of \citet{ramsay2022concentration}. 
This is supported by the numerical experiments we completed, see Section~\ref{sec::sim}, where we observed the ``outlier error amplification'' effect. 
\end{example}
\section{Numerical experiments}\label{sec::sim}
To complement our theoretical results, we provide a numerical demonstration of our proposed methods. 
In order to compute the private projection-depth-based median, we employ discretized Langevin dynamics (as was done by \citet{ramsay2022concentration}) and compute outlyingness by (randomly) discretizing the unit sphere, see Appendix~\ref{app::algorithm} for the algorithm. 
More efficient algorithms could be constructed using the methods and results of \citet{DYCKERHOFF2021}, but we leave that for future work. The current implementation's differential privacy guarantee may be affected by the use of Markov chain Monte Carlo (MCMC) methods \citep{mcmcps}. 
The purpose of the implementation is to numerically investigate the cost of privatizing the projection-depth-based median and compare to some other private measures of location, and the non-private projection-depth-based median. 
Our code is available on \href{https://github.com/12ramsake/PTR-PD-Median}{GitHub}. 

We compute the private and non-private projection-depth-based median based on the median, with the median absolute deviation as the scale estimate. 
We compare the performance of the projection-depth-based medians to the private smoothed integrated dual depth median \citep{ramsay2022concentration}, the non-private sample mean, and the private coin-press mean of \citet{Biswas2020}. 
Following \citet{ramsay2022concentration}, for the integrated dual depth median, we set the smoothing parameter to 10, the prior as $\cN(0,25dI)$. For the coin-press mean, we set the bounding ball to have a radius of $10\sqrt{d}$ and the algorithm was run for four iterations. 

\begin{figure}[t]
\includegraphics[width=\textwidth]{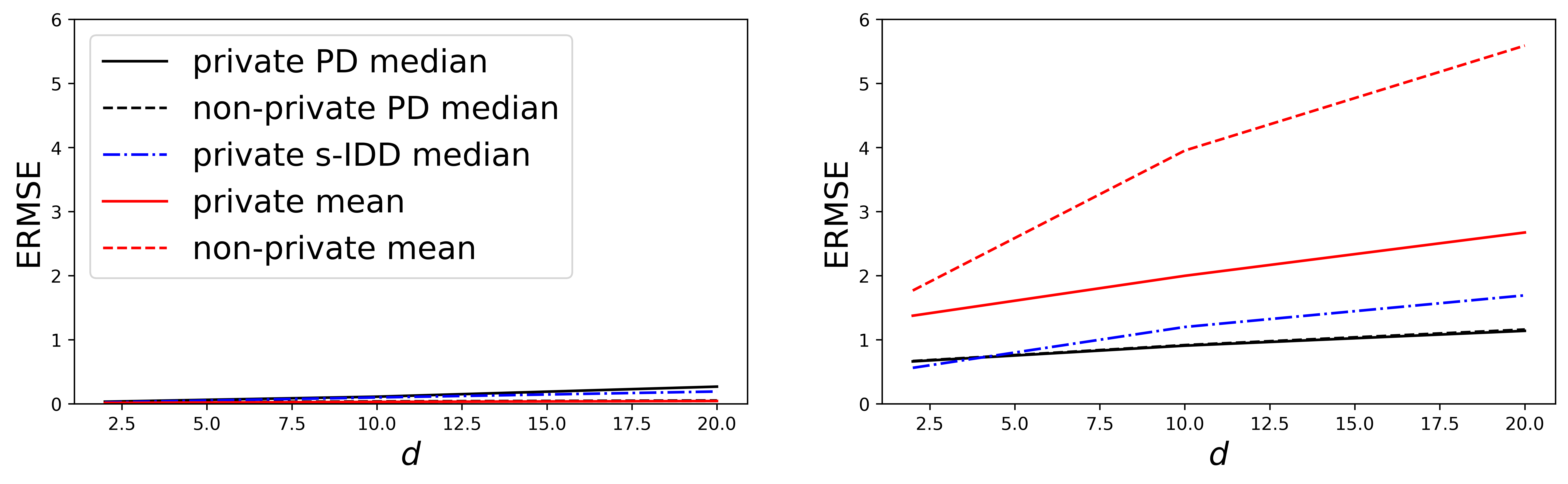}
    \caption{Empirical root mean squared error (ERMSE) of the location estimates under (left) Gaussian data and (right) contaminated Gaussian data. The cost of privacy is eclipsed by the outlier error amplification. The private estimates are all relatively close to their non-private counterparts. By contrast, the increase in error for all estimators resulting from contamination is apparent. The usual differences between robust and non-robust estimators are also apparent. Notably, the projection-depth-based median is the most robust.}
    \label{fig:dimension}
\end{figure}
Following \citet{ramsay2022concentration}, we simulated fifty instances of each PD median, with $n=10,000$ over a range of dimensions from two to twenty. 
The data were either generated from a standard Gaussian measure or contaminated standard Gaussian measure. 
In the contaminated model, 25\% of the observations had mean $(5,\ldots,5)$. 
We set $\epsilon=10,\ \delta=10/n$, $\tau=1$ and $\eta=30(\log n)/n$. The number of vectors was chosen to be 500 for $d<20$ and 1000 otherwise. 

Figure~\ref{fig:dimension} displays the empirical root mean squared error (ERMSE) for the different location estimates. 
It is apparent that the cost of privacy is relatively low for all estimates, though, the classical trade-off between robustness and efficiency is evident. 
The mean estimators perform best in the uncontaminated setting. The opposite is true in the contaminated setting. 
The outlier error amplification is greater than the cost of privacy in both scenarios. 
That is, the increase in error in all estimates resulting from the outliers is much larger than the increase in error from private estimation. 
We also note that the private smoothed integrated dual depth median performs better than the private projection-depth-based median in the uncontaminated setting, but worse than the private projection-depth-based median in the contaminated setting. 
When considering which median to use, it is helpful to note that the smoothed integrated dual depth median is considerably faster to compute in high dimensions than the projection-depth-based median.
However, it is only equivariant under similarity transformations \citep{ramsay2022concentration}, while the projection-depth-based medians, are equivariant under affine transformations \citep{Zuo2003}. 

\section{General results on propose-test-release mechanisms}\label{sec::ptr_general}

In this section, we elaborate on the notion of safety \citep{Brown2021, 2021Liub} in the general case of PTR mechanisms. 
We show that safety is both sufficient and necessary for privacy of a PTR mechanism. We also present general results for analyzing a PTR mechanism based on the exponential mechanism. 
We begin by introducing a general notion of a test mechanism. 
\begin{definition}\label{dfn::test_m}
    A test mechanism is any mechanism, $Q$, such that for all $\hat\nu\in \widehat\cM_{1}(n,d)$,  $Q_{\hat\nu}$ is a Bernoulli measure with probability of success $\lambda(\hat\nu)$, for some $0\leq \lambda(\hat\nu)\leq 1$.
\end{definition}
\noindent Definition \ref{dfn::test_m} simply says that $Q$ maps an empirical measure $\hat\nu$ to a Bernoulli measure with probability of success $\lambda(\hat\nu)$, for some $0\leq \lambda(\hat\nu)\leq 1$. 
Note that every test mechanism is characterized by its associated $\lambda$ function. 
\begin{definition}\label{dfn::ptr_gen}
Given a mechanism $Q$ whose range is $\cM_1(\Theta)$ and a test mechanism characterized by $\lambda$, the PTR mechanism based on $(Q,\lambda)$ is given by the map
$$P_{\hat\nu}\colon (B)\mapsto\lambda(\hat\nu)Q_{\hat\nu}(B\backslash \{\perp\}) + \left[1-\lambda(\hat\nu)\right]\ind{\perp\in B},$$
where $B\in\sB(\Theta\cup \{\perp\})$ and $\hat\nu\in \widehat{\cM}_1(n,d)$. 
\end{definition}
The probability of any event under the PTR mechanism $P$ applied to $\hat\nu$ will be the probability that the test is passed, multiplied by the probability of the event under the mechanism $Q$ applied to $\hat\nu$, plus the probability of the test failing if $\perp$ is an element of the event. 
We now introduce generalized definitions of safety and the safety margin, which apply to arbitrary mechanisms. Subsequently, we will introduce a generalized definition of the exponential mechanism, which applies to an arbitrary cost function $\phi$.

\begin{definition}\label{dfn::sfty_gen}
Given $\epsilon,\delta>0$, the pair $(Q,\hat\nu)$ is $(\epsilon,\delta)$-safe if, for all $\tilde\nu\in\widetilde\cM(\hat\nu,1)$ and measurable sets $B$, it holds that
$$Q_{\hat{\nu}}(B )\leq e^{\epsilon}Q_{\tilde{\nu}}(B)+\delta\qquad\text{and}\qquad Q_{\tilde{\nu}}(B )\leq e^{\epsilon}Q_{\hat{\nu}}(B )+\delta.$$
The $(\epsilon,\delta)$-safe set of $Q$ is
$$\bfS(Q,\epsilon,\delta)= \{ \hat\nu\in \widehat{M}_1(n,d)\colon\  (Q,\hat\nu) \text{ is } (\epsilon,\delta)\text{-safe} \}.$$
\end{definition}
Observe that Definition \ref{dfn::sfty_gen} is a generalized version of Definition \ref{dfn::sfty}. 
Definition \ref{dfn::sfty} is a special case of Definition \ref{dfn::sfty_gen}, with $Q=\EM$. 
Note that $Q$ is allowed to depend on $\epsilon$ and $\delta$. 
Safety is linked to the privacy of a given PTR mechanism. 
Colloquially, if $\lambda(\hat\nu)$ is large, then the dataset must satisfy some level of safety. 
In addition, sufficient conditions for differential privacy of a PTR mechanism are that the test mechanism is differentially private and that $\lambda(\hat\nu)$ is very small if $\hat\nu$ is unsafe. 
\begin{definition}\label{dfn::sfty_gen_mar}
Given $\epsilon,\delta>0$ and a mechanism $Q$, the safety margin of an empirical measure $\hat\nu$ with respect to $Q$, is 
$$\SM(\hat\nu,Q; \epsilon,\delta)=\inf\{m\in \mathbb{Z}^+\colon\ \exists \tilde\nu\in\tilde{\cM}(\hat\nu,m)\ s.t.\ \tilde\nu\notin \bfS(Q,\epsilon,\delta) \}.$$
\end{definition}
Observe that Definition \ref{dfn::sfty_gen_mar} is a generalized version of Definition \ref{dfn::sfty_mar}. 
That is, Definition \ref{dfn::sfty_mar} is a special case of Definition \ref{dfn::sfty_gen_mar} with $Q$ always set to $\EM$. 
When $\epsilon$ and $\delta$ are clear from the context, we omit them, denoting the safety margin as $\SM(\hat\nu,Q).$ 
If $\lambda(\hat\nu)$ is small whenever $\hat\nu$ is not in the safe set of $Q$ and the test mechanism is differentially private, then a PTR mechanism will be differentially private.
\begin{theorem}\label{thm::s_suff}
Given $\epsilon_1,\epsilon_2,\delta_{1},\delta_{2},\delta_3>0$, let $P$ be the PTR mechanism based on $Q$ and  $\lambda$. 
If $\lambda$ is $(\epsilon_1,\delta_{1})$-differentially private and is such that for any $\hat\nu\notin \bfS(Q,\epsilon_2,\delta_{2})$, $\lambda(\hat\nu,\epsilon_1,\delta_{1})<\delta_3$, then $P$
is $(\epsilon_1+\epsilon_2,\delta_{1}+\delta_{2}\vee \delta_3)$-differentially private. 
\end{theorem}
\begin{proof}
First, consider $\hat\nu\in \bfS(Q,\epsilon_2,\delta_{2})$. By composition, the differential privacy bound with parameters $(\epsilon_1+\epsilon_2,\delta_{1}+\delta_{2})$ will be satisfied. 
Now, consider $\hat\nu\notin \bfS(Q,\epsilon_2,\delta_{3})$. 
For sets such that $\perp\notin A$, then $P_{\hat\nu}(A)\leq \lambda(\hat\nu)\leq \delta_3.$
For sets such that $\perp\in A$, then 
\begin{equation*}
   P_{\hat\nu}(A)\leq P_{\hat\nu}(\{\perp\})+\delta_3\leq P_{\tilde\nu}(\{\perp\})e^{\epsilon_1}+\delta_{2}+\delta_3. \qedhere
\end{equation*}
\end{proof}
The previous theorem gives sufficient conditions on both the privacy guarantee of the test mechanism and the relationship between the test mechanism and the safe set of $Q$ to gain differential privacy of a PTR mechanism. 
We can conversely show that a $(\epsilon,\delta)$-differentially private propose test release mechanism consists of a test that is at least $(\epsilon,\delta)$-differentially private and, for datasets likely to pass the test, the release mechanism-dataset pair is safe. 
\begin{theorem}
Suppose that a PTR mechanism $P$ based on $Q$ and $\lambda$ is $(\epsilon,\delta)$-differentially private. 
Then the follow statements hold:
\begin{enumerate}
    \item The test mechanism characterized by $\lambda$ satisfies $(\epsilon,\delta)$-differential privacy.
    \item For any $\hat{\nu}$ such that $\lambda(\hat{\nu})>1/k$, the pair $(Q,\hat\nu)$ is $(2\epsilon,k\delta e^\epsilon)$-safe.
\end{enumerate}
\end{theorem}
\begin{proof}
    The first item follows from applying the  $(\epsilon,\delta)$-differential privacy inequality to the set $\{\perp\}$. 
    The second item follows the first item, and from applying the $(\epsilon,\delta)$-differential privacy inequality to a set $A$ such that $\perp\notin A$. 
\end{proof}


We now consider PTR mechanisms based on general exponential mechanisms. 
For a general cost function $\phi(\cdot;\hat{\nu})$, let $A_{\tau,\hat{\nu}}=\{\phi(x ,\hat{\nu})\leq \tau\}$. 
Then, we can define the exponential mechanism $\EM_{\hat{\nu},\phi}(\cdot;\tau,\eta,\epsilon)$ as the measure whose density satisfies
\begin{equation*}\label{eqn::EM}
d\EM_{\hat{\nu},\phi}(\cdot;\tau,\eta,\epsilon)\propto\ind{x\in A_{\tau,\hat{\nu}}}\exp(-\phi(x,{\hat{\nu})\epsilon/4\eta})dx.
\end{equation*}

Now, the critical components for accuracy are that $\Prr{\hat{\nu}\in \bfS(Q,\epsilon,\delta)}$ is large and that $\lambda(\hat{\nu})$ is large when $\hat{\nu}\in \bfS(Q,\epsilon,\delta)$. 
We formalize this as follows. 
Let $G$ be the CDF of a standard Laplace RV. 
\begin{lemma}\label{lem::gen_bound_no_reply}
For all $\epsilon,\tau,\eta>0$, $\delta\leq e^{-\epsilon/2}/2$, $n,d\geq 1$ and $\nu\in \cM_1(\rdd)$ it holds that
\begin{multline*}
        \E{}{\lambda(\hat\nu,\epsilon,\delta)}\geq 1-\Prr{\inf_{y>0}e^{-\epsilon y/4\eta} \frac{\vol(A_{\tau+2\eta,\hat\nu}\cap A_{\tau-2\eta,\hat\nu}^c)}{\vol(A_{\tau-y-2\eta,\hat{\nu}})}> e^{-\epsilon/4}\delta/3}\\
        -\inf_{k\in \{0\}\cup [n-1]} \left[ \Prr{    \sup_{\substack{\tilde{\nu}\in \am{\hat\nu}{k+1}\\x\in A_{\tau+\eta,\hat{\nu}}\cup A_{\tau+\eta,\tilde{\nu}}}}|\phi(x,\hat{\nu})-\phi(x,\tilde{\nu})|>\eta/2}+G(\log(1/2\delta)-\epsilon k/2)\right].
\end{multline*}
\end{lemma}
The proof of Lemma~\ref{lem::gen_bound_no_reply} can be found in Appendix~\ref{app::gen_proofs}. 
This has the interpretation that the probability of passing the test is high when the probability that, for some small $\alpha>0$, the $\tau\pm \alpha$-contours of $\phi(\cdot,\hat\nu)$ are not too steep and the objective is locally insensitive within the $\tau\pm \alpha$-contour. 


\bibliographystyle{apalike}
\bibliography{main.bib}
\appendix

\section{Algorithm}\label{app::algorithm}
Here we outline the algorithm used to compute the private, projection-depth-based median in Section~\ref{sec::sim}. 
We will approximate $O(x,\hat\nu)$ by $\widehat O(x,\hat\nu)=\sup_{u\in \cU_N} O_u(x,\hat\nu)$, where $\cU_N$ is a set of $N$ unit vectors sampled uniformly from $\bS^{d-1}$. 
The algorithm consists of two parts, executing the test and sampling from the exponential mechanism. 
For the latter, observe that the outlyingness function is quasiconvex, i.e., its lower level sets are convex. 
This allows us to apply an envelope theorem to compute the gradient of $\widehat O_u(x,\hat\nu)$ with respect to $x$. To this end, for $x\in A_\tau$, let $u_x$ be the unit vector that maximizes $\widehat O_u(x,\hat\nu)$. 
Then, for any $x\in A_\tau$, it holds that
\begin{align*}
    \nabla \widehat O(x,\nu)&=u_x\frac{\mathrm{sign}(x^\top u_x-\mu_{u_x} )}{\sigma_{u_x}}. 
\end{align*}
To compute the non-private projection-depth-based median, we use gradient descent, and, to sample from the exponential mechanism, we use Langevin Dynamics. 
Given that Langevin Dynamics samples approximately from the exponential mechanism, we note that this may violate the differential privacy condition, and stress that this implementation is just to confirm the theoretical properties of the proposed private median. 
Of course, more sophisticated methods to compute the outlyingness could be developed, such as those based on \citep{DYCKERHOFF2021}, but we leave that for future work. Indeed, given the caveat on using Langevin Dynamics to sample from the exponential mechanism mentioned above, an algorithm to compute the private projection-depth-based median is an interesting open problem. 
\begin{algorithm}[t]
\setstretch{1.2}
\caption{Exponential Mechanism via Langevin Dynamics}
\label{alg:EM}
\begin{algorithmic}
\Procedure{EM}{$X_1,\ldots,X_n,\epsilon,\eta,\tau,N,T,\omega$}
\State Set $\cU_N$ to be $N$ independent random vectors uniformly over $\bS^{d-1}$. 
\State Set $Y_0=\cN(0,I)$
\For{$i\in [T]$}
\State Set $O=\widehat O(x,\hat\nu)$
\If{$O>\tau$}
\State Set $\nabla O=0$
\Else
\State Set $\nabla O=\nabla \widehat O(x,\hat\nu)$
\EndIf
\State Set $Z_i=\cN(0,I)$
\State Set $Y_i=Y_{i-1}-\omega\frac{\epsilon}{4\eta}\nabla O+\sqrt{2\omega}Z_i$
\EndFor\\
\Return $Y_{T}$
\EndProcedure
\end{algorithmic}
\end{algorithm}
\begin{algorithm}[t]
\setstretch{1.2}
\caption{Test Mechanism}
\label{alg:Test}
\begin{algorithmic}
\Procedure{Test}{$X_1,\ldots,X_n,\epsilon,\delta,\eta,\tau,d$}
\State Set $k=1$
\State Set $\mathrm{Dev\_Cond}=$True
\State Set $\mathrm{Vol\_Cond}=$True
\While{$\mathrm{Dev\_Cond}$ and $\mathrm{Vol\_Cond}$}
\State Set $k=k+1$
\State Set $\mathrm{Vol\_Cond}=\mathrm{VR}_k<\delta$
\State Compute $\widehat S_{n,k}(\sigma),\ \widehat S_{n,k}(\mu),\ \hat b$
\State Set $\mathrm{Dev\_Cond}=\widehat \Delta_k<\eta$
\EndWhile
\State Set $\SM=k-2$. 
\State Generate $W\sim\lap(0,1)$\\
\Return $\SM+\frac{2}{\epsilon}>\frac{2\log(1/2\delta)}{\epsilon}$
\EndProcedure
\end{algorithmic}
\end{algorithm}
\begin{algorithm}[t]
\setstretch{1.2}
\caption{Private PDB Median}
\label{alg:o_can}
\begin{algorithmic}
\Procedure{Private\_Median}{$X_1,\ldots,X_n,\epsilon,\delta,\eta,\tau,d,T,\omega$}
\State Set $\mathrm{test}=$Test($X_1,\ldots,X_n,\epsilon,\delta,\eta,\tau,d$)
\If{$\mathrm{test}$}
\Return $\perp$
\Else{} \ 
\Return $\EM(X_1,\ldots,X_n,\epsilon,\eta,\tau,T,\omega)$
\EndIf
\EndProcedure
\end{algorithmic}
\end{algorithm}

To perform the test, it is clearly difficult to compute the safety margin in this case. 
As a result, we can make use of Lemma~\ref{thm::safety_margin_lower_bound}, Lemma~\ref{lem::ls_out} and Lemma~\ref{lem:vr_bound}.
Now, define $$\widehat{S}_{n,k}(T)=\sup_{\substack{\tilde{\nu}\in \am{\hat\nu}{k}\\ u\in \cU_N}}|T(\hat\nu_{u})-T(\tilde{\nu}_u)|.$$
We must first compute three quantities: $\widehat S_{n,k}(\mu),\widehat S_{n,k}(\sigma)$ and $\hat b=\inf_{\substack{\tilde{\nu}\in \am{\hat\nu}{k}\cup\{\hat\nu\}\\ u\in \cU_N}}\sigma(\hat\nu_u)$. 
First, it is easy to check that if $\mu=\med$, then
$$\widehat S_{n,k}(\mu)\leq \sup_{u\in \cU_N} [|(X^\top u)_{(\floor{(n+1)/2}+k)}-(X^\top u)_{(\floor{(n+1)/2})}|\vee |(X^\top u)_{(\floor{(n+1)/2})}-(X^\top u)_{(\floor{(n+1)/2}-k)}|].$$


In order to compute the $\widehat S_{n,k}(\sigma)$ where $\sigma=\mad$, we can first consider the local sensitivity of the median absolute deviation, for a univariate sample $Z_1,\ldots,Z_n$ with empirical measure $\hat\nu$. 
Let $Y_{i}(z)=|Z_i-z|$ and call this measure $\hat\nu_z$. 
Consider $\tilde\nu\in \cM(k,\hat\nu)$. 
First, note that $\mu(\tilde\nu)\in [Z_{(\floor{(n+1)/2}+k)},Z_{(\floor{(n+1)/2}-k)}]=[a_k,b_k]$. 
Now observe that
\begin{align*}
    \sup_{\tilde\nu\in \cM(k,\hat\nu)}|\sigma(\tilde\nu)-\sigma(\hat\nu)|&\leq \sup_{z\in [a_k,b_k]}\sup_{\tilde\nu\in \cM(k,\hat\nu_z)}|\med(\hat\nu_z)-\med(\tilde\nu)|\\
    &\leq \sup_{z\in [a_k,b_k]}|Y_{(\floor{(n+1)/2}+k)}(z)-Y_{(\floor{(n+1)/2})}(z)|\\
    &\hspace{100pt}\vee |Y_{(\floor{(n+1)/2}-k)}(z)-Y_{(\floor{(n+1)/2})}(z)|. 
\end{align*}
Note that the right-hand side of the above inequality is a piecewise linear function in $z$, with knots at the observations. 
Thus, it suffices to compute this function for the $2k$ observations in $[a_k,b_k]$. 
Thus, letting $Y_{i,u}(z)=|X_i^\top u-z|$,
$$\widehat S_{n,k}(\sigma)\leq \sup_{\substack{u\in \cU_N\\ z\in [a_k,b_k]}} |Y_{(\floor{(n+1)/2}+k),u}(z)-Y_{(\floor{(n+1)/2}),u}(z)|\vee |Y_{(\floor{(n+1)/2}-k),u}(z)-Y_{(\floor{(n+1)/2}),u}(z)|.$$
Similarly,
$\hat b \leq  \inf_{u\in \cU_N}\inf_{ z\in [a_k,b_k]}|Y_{(\floor{(n+1)/2}),u}(z)|.$ 

Next, we can compute $\widehat \Delta_k=((\tau+\eta) \widehat S_{n,k}(\sigma)+\widehat S_{n,k}(\mu))/\hat b$. 
The last step is to compute the volume condition in Lemma~\ref{thm::safety_margin_lower_bound} using Lemma~\ref{lem:vr_bound} and check if 
\begin{equation*}
\mathrm{VR}_k\coloneqq\exp\left(-\epsilon \tau/4\eta+\epsilon(1/2+\hat\alpha_{\mu}/4\inf_{\ubsd} \hat\sigma_u\eta)+d\right)\left(\frac{\sup_{\ubsd} \hat\sigma_u (\tau+2\eta)+\hat\alpha_{\mu}}{\inf_{\ubsd} \hat\sigma_u4\eta d\epsilon}\right)^d\leq \delta.
\end{equation*}
All the described methods are summarized in Algorithms~\ref{alg:EM}--\ref{alg:o_can}.
\section{Proofs from Section~\ref{sec::PD}}\label{app:proofs}
We prove the main theorems with a series of lemmas. 
Let $\hat\alpha_{\mu}=\sup_{\ubsd}\hat\mu_u-\inf_{\ubsd}\hat\mu_u$ and $\alpha_{\mu}=\sup_{u\in\bS^{d-1}}\mu_u-\inf_{\ubsd}\mu_u$. 
The first lemma bounds the values of $\mu$ and $\sigma$ when computed on contaminated datasets. 
\begin{lemma}
    \begin{enumerate}
        \item Assuming Conditions~\ref{cond::LS_Bound}-\ref{cond::bounded-parameters}: \begin{equation}
            \Prrr{\inf_{\substack{u\in \bS^{d-1}\\\tilde{\nu}\in \{\hat\nu_u\}\cup\am{\hat\nu_u}{k}}} \sigma(\tilde\nu) \leq \frac{c_1}{2}} \lesssim a_{1}e^{-a_{2}nc_1/4}+\kappa_{n}+\zeta_{n,\sigma}(\frac{c_1}{4}). \label{eq::E2C_bound}
        \end{equation}
    \item Assuming $\alpha_\mu \leq 2c_3$ and that Condition~\ref{cond::Pop_Bound} holds: \begin{equation}
        \Prr{\hat\alpha_\mu > 4c_3} \lesssim \zeta_{n,\mu}(c_3). \label{eq::E3C_bound}
    \end{equation}
    \item Assuming Conditions~\ref{cond::Pop_Bound} and~\ref{cond::bounded-parameters}: \begin{equation}
        \Prr{\sup_{u\in \bS^{d-1}}\sigma(\hat\nu_{u})\geq 2c_2}\lesssim \zeta_{n,\sigma}(c_2).\label{eq::E4C_bound}
    \end{equation}
    \item Assuming Conditions~\ref{cond::Pop_Bound} and~\ref{cond::bounded-parameters}:\begin{equation}
        \Prr{\inf_{u\in \bS^{d-1}}\sigma(\hat\nu_{u})\leq c_1/2} \lesssim \zeta_{n,\sigma}(\frac{c_1}{2}).\label{eq::E7C_bound}
    \end{equation}
    
    \end{enumerate}
\end{lemma}
\begin{proof}
\textbf{Proof of \eqref{eq::E2C_bound}:} \begin{align*}
    \inf_{\substack{u\in \bS^{d-1}\\\tilde{\nu}\in \{\hat\nu_u\}\cup\am{\hat\nu_u}{k}}} \sigma(\tilde\nu) &= \inf_{\substack{u\in \bS^{d-1}\\\tilde{\nu}\in \{\hat\nu_u\}\cup\am{\hat\nu_u}{k}}} (\sigma(\tilde\nu) - \sigma(\hat\nu) + \sigma(\hat\nu)) \\
    &\geq \inf_{\substack{u\in \bS^{d-1}\\\tilde{\nu}\in \{\hat\nu_u\}\cup\am{\hat\nu_u}{k}}} (\sigma(\tilde\nu) - \sigma(\hat\nu)) + \inf_{u\in \bS^{d-1}}\sigma(\hat\nu) \\
    &\geq \inf_{\substack{u\in \bS^{d-1}\\\tilde{\nu}\in \{\hat\nu_u\}\cup\am{\hat\nu_u}{k}}} (\sigma(\tilde\nu) - \sigma(\hat\nu)) + \inf_{u\in \bS^{d-1}}(\sigma(\hat\nu) - \sigma(\nu)) + \inf_{u\in \bS^{d-1}}\sigma(\nu)\\ 
    &\geq -\sup_{\substack{u\in \bS^{d-1}\\\tilde{\nu}\in \{\hat\nu_u\}\cup\am{\hat\nu_u}{k}}} |\sigma(\tilde\nu) - \sigma(\hat\nu)| - \sup_{u\in \bS^{d-1}}|\sigma(\hat\nu) - \sigma(\nu)| + c_1.
\end{align*}

The last inequality uses Condition~\ref{cond::bounded-parameters}. Then, if $\inf_{\substack{u\in \bS^{d-1}\\\tilde{\nu}\in \{\hat\nu_u\}\cup\am{\hat\nu_u}{k}}} \sigma(\tilde\nu) \leq c_1/2$, \begin{align*}
    -\sup_{\substack{u\in \bS^{d-1}\\\tilde{\nu}\in \{\hat\nu_u\}\cup\am{\hat\nu_u}{k}}} |\sigma(\tilde\nu) - \sigma(\hat\nu)| - \sup_{u\in \bS^{d-1}}|\sigma(\hat\nu) - \sigma(\nu)| + c_1 &\leq \frac{c_1}{2} \\
    \implies \sup_{\substack{u\in \bS^{d-1}\\\tilde{\nu}\in \{\hat\nu_u\}\cup\am{\hat\nu_u}{k}}} |\sigma(\tilde\nu) - \sigma(\hat\nu)| + \sup_{u\in \bS^{d-1}}|\sigma(\hat\nu) - \sigma(\nu)| &\geq \frac{c_1}{2}.
\end{align*} Taken together, \begin{align*}
    \Prr{\sup_{u\in \bS^{d-1}}\sigma(\hat\nu_{u})> 2c_2} &\leq \Prr{\sup_{\substack{u\in \bS^{d-1}\\\tilde{\nu}\in \{\hat\nu_u\}\cup\am{\hat\nu_u}{k}}} |\sigma(\tilde\nu) - \sigma(\hat\nu)| + \sup_{u\in \bS^{d-1}}|\sigma(\hat\nu) - \sigma(\nu)| \geq \frac{c_1}{2}} \\
    &\leq \Prr{\sup_{\substack{u\in \bS^{d-1}\\\tilde{\nu}\in \{\hat\nu_u\}\cup\am{\hat\nu_u}{k}}} |\sigma(\tilde\nu) - \sigma(\hat\nu)|\geq \frac{c_1}{4}} \\
    &\hspace{10em} + \Prr{\sup_{u\in \bS^{d-1}}|\sigma(\hat\nu) - \sigma(\nu)| \geq \frac{c_1}{4}} \\
    &\lesssim a_{1}e^{-a_{2}nc_1/4}+\kappa_{n}+\zeta_{n,\sigma}(\frac{c_1}{4}).
\end{align*}
The final inequality follows directly from Conditions~\ref{cond::LS_Bound} and~\ref{cond::Pop_Bound}.

\noindent\textbf{Proof of \eqref{eq::E3C_bound}:} Note that \begin{align*}
    \hat\alpha_{\mu} = \sup_{u\in\bS^{d-1}}\hat\mu_u-\inf_{u\in\bS^{d-1}}\hat\mu_u  &= \sup_{u\in\bS^{d-1}}(\hat\mu_u + \mu_u - \mu_u) + \sup_{u\in\bS^{d-1}}(-\hat\mu_u + \mu_u - \mu_u) \\
    &\leq  \sup_{u\in\bS^{d-1}}|\hat\mu_u - \mu_u| +  \sup_{u\in\bS^{d-1}} \mu_u + \sup_{u\in\bS^{d-1}}|\mu_u -\hat\mu_u| - \inf_{u\in\bS^{d-1}}\mu_u \\
    &= 2\sup_{u\in\bS^{d-1}}|\mu_u -\hat\mu_u| + a_\mu \\
    &\leq 2\left(\sup_{u\in\bS^{d-1}}|\mu_u -\hat\mu_u| + c_3\right),
\end{align*} 
where we used the fact that $\alpha_\mu \leq 2c_3$. As a result, applying Condition~\ref{cond::Pop_Bound}, 
\begin{align*}
    \Prr{\hat\alpha_{\mu}\geq 4c_3} &\leq \Prr{2\left(\sup_{u\in\bS^{d-1}}|\mu_u -\hat\mu_u| + c_3\right) > 4c_3}\lesssim \zeta_{n,\mu}(c_3).
\end{align*}
\noindent \textbf{Proof of \eqref{eq::E4C_bound}:} \begin{align*}
    \sup_{u\in \bS^{d-1}}\sigma(\hat\nu_{u}) &= \sup_{u\in \bS^{d-1}}\left(\sigma(\hat\nu_{u})-\sigma(\nu_u)+\sigma(\nu_u)\right) \\
    &\leq \sup_{u\in \bS^{d-1}}\left(\sigma(\hat\nu_{u})-\sigma(\nu_u)\right)+\sup_{u\in \bS^{d-1}}\sigma(\nu_u) \\
    &\leq \sup_{u\in \bS^{d-1}}|\sigma(\hat\nu_{u})-\sigma(\nu_u)|+c_2,
\end{align*}
where the final inequality relies on condition~\ref{cond::bounded-parameters}. Then, \begin{align*}
    \Prr{\sup_{u\in \bS^{d-1}}\sigma(\hat\nu_{u})\geq 2c_2} &\leq \Prr{ \sup_{u\in \bS^{d-1}}|\sigma(\hat\nu_{u})-\sigma(\nu_u)|+c_2 \geq 2c_2} \\
    &= \Prr{ \sup_{u\in \bS^{d-1}}|\sigma(\hat\nu_{u})-\sigma(\nu_u)| \geq c_2} \\
    &\lesssim \zeta_{n,\sigma}(c_2),
\end{align*} where the final result is a straightforward application of condition~\ref{cond::Pop_Bound}.

\noindent \textbf{Proof of \eqref{eq::E7C_bound}:} \begin{align*}
    \inf_{u\in \bS^{d-1}}\sigma(\hat\nu_{u}) &= \inf_{u\in \bS^{d-1}}\left(\sigma(\hat\nu_{u})-\sigma(\nu_u) + \sigma(\nu_u)\right) \\
    &\geq \inf_{u\in \bS^{d-1}}\left(\sigma(\hat\nu_{u})-\sigma(\nu_u)\right) + \inf_{u\in \bS^{d-1}}\left(\sigma(\nu_u)\right) \\
    &= -\sup_{u\in \bS^{d-1}}\left(\sigma(\nu_u)-\sigma(\hat\nu_{u})\right) + \inf_{u\in \bS^{d-1}}\left(\sigma(\nu_u)\right) \\
    &\geq -\sup_{u\in \bS^{d-1}}\left|\sigma(\nu_u)-\sigma(\hat\nu_{u})\right| + c_1,
\end{align*} where the final inequality relies on Condition~\ref{cond::bounded-parameters}. This means that \begin{align*}
    \Prr{\inf_{u\in \bS^{d-1}}\sigma(\hat\nu_{u}) \leq \frac{c_1}{2}} &\leq \Prr{-\sup_{u\in \bS^{d-1}}\left|\sigma(\nu_u)-\sigma(\hat\nu_{u})\right| + c_1 \leq \frac{c_1}{2}} \\
    &= \Prr{\sup_{u\in \bS^{d-1}}\left|\sigma(\nu_u)-\sigma(\hat\nu_{u})\right| \geq \frac{c_1}{2}}\\
    &\lesssim \zeta_{n,\sigma}(\frac{c_1}{2}),
\end{align*} where the final line is a simple application of Condition~\ref{cond::Pop_Bound}.
\end{proof}
The next lemma bounds the local $k$-sensitivity of the outlyingness function. 
Let $b_{k,u}^-= \inf_{\tilde{\nu}\in \{\hat\nu_u\}\cup\am{\hat\nu_u}{k}}\sigma(\tilde\nu)$. 
\begin{lemma}\label{lem::ls_out}
For all $\tau>0$, $n,d\geq 1$, $\hat\nu\in\widehat\cM(n,d)$ and $k\in[n]$ 
\begin{equation}
 \sup_{\substack{\tilde\nu\in\am{\hat\nu}{k}\\
 x\in A_{\tau,\tilde\nu}\cup A_{\tau,\hat\nu}}}|O(x,\hat{\nu})-O(x,\tilde{\nu})|\leq \frac{\tau  \tilde{S}_{n,k}(\sigma)+\tilde{S}_{n,k}(\mu)}{  \inf_{u\in\bS^{d-1}} b_{k,u}^-} \coloneqq \Delta_k.
\end{equation}
\end{lemma}
\begin{proof}
We use the fact that 
$$\sup_{\substack{\tilde\nu\in\am{\hat\nu}{k}\\
 x\in A_{\tau,\tilde\nu}\cup A_{\tau,\hat\nu}}}|O(x,\hat{\nu})-O(x,\tilde{\nu})|\leq \sup_{\substack{\tilde\nu\in\am{\hat\nu}{k}\\ u\in \bS^{d-1}\\ x\in A_{\tau,\tilde\nu}\cup A_{\tau,\hat\nu}}}|O_u(x,\hat{\nu})-O_u(x,\tilde{\nu})|,$$
and bound the right-hand side. 
Observe that for all $\tau>0$ and all $\tilde\nu\in\am{\hat\nu}{k}$ it holds that
\begin{align*}
  \sup_{\substack{u\in \bS^{d-1}\\ x\in A_{\tau,\hat\nu}}}|O_u(x,\hat{\nu})-O_u(x,\tilde{\nu})|&= \sup_{\substack{u\in \bS^{d-1}\\ x\in A_{\tau,\hat\nu}}}\left|\frac{x^\top u-\hat{\mu}_{u}}{\hat{\sigma}_u}-\frac{x^\top u-\mu(\tilde{\nu}_u)}{\sigma(\tilde{\nu}_u)}\right| \\
 &\leq \sup_{u\in\bS^{d-1}}\frac{1}{\hat{\sigma}_u\sigma(\tilde{\nu}_u)}\left(\left( \sup_{x\in A_{\tau,\tilde\nu}}|x^\top u-\hat\mu_u| \right)\tilde{S}_{n,k}(\sigma)+  \sigma(\hat{\nu}_u) \tilde{S}_{n,k}(\mu)\right)\\
  &\leq \frac{\tau  \tilde{S}_{n,k}(\sigma)+\tilde{S}_{n,k}(\mu)}{  \inf_{u\in\bS^{d-1}}b^-_{k,u} }.
\end{align*}
Using the same logic, for all $\tau>0$ and all $\tilde\nu\in\am{\hat\nu}{k}$ we can write:
\begin{align*}
 \sup_{\substack{u\in \bS^{d-1}\\ x\in A_{\tau,\tilde\nu}}}|O_u(x,\hat{\nu})-O_u(x,\tilde{\nu})| &\leq \frac{\tau  \tilde{S}_{n,k}(\sigma)+\tilde{S}_{n,k}(\mu)}{ \inf_{u\in\bS^{d-1}} \hat\sigma_u }\leq \frac{\tau  \tilde{S}_{n,k}(\sigma)+\tilde{S}_{n,k}(\mu)}{\inf_{u\in\bS^{d-1}}b^-_{k,u}}.
\end{align*}
Combining these two results yields that, for all $\tau>0$ and all $\tilde\nu\in\am{\hat\nu}{k}$, it holds that 
\begin{equation*}
  \sup_{\substack{u\in \bS^{d-1}\\ x\in A_{\tau,\tilde\nu}\cup A_{\tau,\hat\nu}}}|O_u(x,\hat{\nu})-O_u(x,\tilde{\nu})|\leq \frac{\tau  \tilde{S}_{n,k}(\sigma)+\tilde{S}_{n,k}(\mu)}{  \inf_{u\in\bS^{d-1}} b_{k,u}^-} . \qedhere
\end{equation*}
\end{proof}
%
The next lemma bounds the change in volume of $A_\tau$, when $\tau$ is slightly increased. 
\begin{lemma}\label{lem:vr_bound}
For $\phi=O$, and all $\hat\nu\in\widehat\cM_1(n,d)$ and $\tau,\eta,y>0$ such that $\inf_{\ubsd} \hat\sigma_u (\tau-y-2\eta)-\hat\alpha_{\mu}>0$, it holds that
{\small
\begin{equation*}
\frac{\vol(A_{\tau+2\eta,\hat\nu}\cap A^c_{\tau-2\eta,\hat\nu})}{\vol(A_{\tau-y-2\eta,\hat{\nu}})}\leq  \left(\frac{\sup_{\ubsd} \hat\sigma_u (\tau+2\eta)+\hat\alpha_{\mu}}{\inf_{\ubsd} \hat\sigma_u (\tau-y-2\eta)-\hat\alpha_{\mu}}\right)^d-\left(\frac{\inf_{\ubsd} \hat\sigma_u (\tau-2\eta)-\hat\alpha_{\mu}}{\inf_{\ubsd} \hat\sigma_u (\tau-y-2\eta)-\hat\alpha_{\mu}}\right)^d.
\end{equation*}}
\end{lemma}
\begin{proof}
Let $a(t)=\inf_{\ubsd} \hat\sigma_u t-\hat\alpha_{\mu}$ and $b(t)=\sup_{\ubsd} \hat\sigma_u t+\hat\alpha_{\mu}$.
We first show that 
$\vol(B_{a(t)})\leq \vol(A_{t,\hat\nu})\leq\vol(B_{b(t)}).$ 
Beginning with the upper bound, for any $z\in\rdd$,
{\small
\begin{align*}
   A_{t,\hat\nu}&= \{x\colon \sup_{\ubsd}\frac{|x^\top u-\hat\mu_u|}{\hat\sigma_u}\leq t\}\subset \{x\colon \sup_{\ubsd}|x^\top u-z^\top u|\leq \sup_{\ubsd} \hat\sigma_u t+\sup_{u\in\bS^{d-1}}|z^\top u-\hat\mu_u|\}.
\end{align*}
}
In particular, taking $z=v\hat\mu_v$ for $v\in\bS^{d-1}$, 
\begin{align*}
   A_{t,\hat\nu}&\subset\{|x^\top u-\hat\mu_v v^\top u|\leq b(t)\}=\{x\colon \norm{x-\hat\mu_v v}\leq b(t)\}=B_{b(t)}(\hat\mu_v v).
\end{align*}
It follows that $\vol(A_{t,\hat\nu})\leq \vol(B_{b(t)}).$
Next, we show the analogous lower bound. To this end, note that for any $z\in\rdd$
\begin{align*}
   A_{t,\hat\nu}&= \{x\colon \sup_{u\in\bS^{d-1}}\frac{|x^\top u-\hat\mu_u|}{\hat\sigma_u}\leq t\}\supset \{x\colon \sup_{u\in\bS^{d-1}}|x^\top u-z^\top u|\leq \inf_{\ubsd} \hat\sigma_u t-\sup_{u\in\bS^{d-1}}|z^\top u-\hat\mu_u|\}.
\end{align*}
Take $z=v\hat\mu_v$, where $v\in\bS^{d-1}$, yielding 
$A_{t,\hat\nu}\supset \{x\colon \norm{x-v\hat\mu_v}\leq a(t)\}=B_{a(t)}(v\hat\mu_v).$ It follows that $\vol(A_{t,\hat\nu})\geq \vol(B_{a(t)}).$
Now, we apply the upper and lower bounds, in conjunction with the assumption that $a(\tau-y-2\eta)=\inf_{\ubsd} \hat\sigma_u (\tau-y-2\eta)-\hat\alpha_{\mu}>0$, which results in
\begin{align*}
    \frac{\vol(A_{\tau+2\eta,\hat\nu}\cap A^c_{\tau-2\eta,\hat\nu})}{\vol(A_{\tau-y-2\eta,\hat{\nu}})}&\leq \frac{\vol(B_{b(\tau+2\eta)}\cap B^c_{a(\tau-2\eta)})}{\vol(B_{a(\tau-y-2\eta)})}\\
    &\leq \left(\frac{\sup_{\ubsd} \hat\sigma_u (\tau+2\eta)+\hat\alpha_{\mu}}{\inf_{\ubsd} \hat\sigma_u (\tau-y-2\eta)-\hat\alpha_{\mu}}\right)^d\\
    &\hspace{1.7in} -\left(\frac{\inf_{\ubsd} \hat\sigma_u (\tau-2\eta)-\hat\alpha_{\mu}}{\inf_{\ubsd} \hat\sigma_u (\tau-y-2\eta)-\hat\alpha_{\mu}}\right)^d. \qedhere
\end{align*}
\end{proof}
We can now prove Theorem~\ref{thm::no-reply}.
\begin{proof}[Proof of Theorem~\ref{thm::no-reply}]
First, note that $\mu$ and $\sigma$ may satisfy Conditions~\ref{cond::LS_Bound} 
for different parameters, e.g., different sequences $\kappa_n$. 
However, we can define a set of joint parameters $\kappa_{n},a_1,a_2$ for which the bound in Condition~\ref{cond::LS_Bound} holds for both $\mu$ and $\sigma$, which we will use throughout the proof. 

An application of Lemma~\ref{lem::gen_bound_no_reply} with $\phi=O$, and subsequently taking $k=2$ yields:
\begin{align}
    \E{}{Z}&\geq 1-\Prr{E_1} \nonumber\\
   & \hspace{30pt}-\inf_{k\in [n]} \left[ \Prrr{\sup_{\substack{\tilde{\nu}\in \am{\hat\nu}{k+1}\\x\in A_{\tau+\eta,\hat{\nu}}\cup A_{\tau+\eta,\tilde{\nu}}}}|O(x,\hat{\nu})-O(x,\tilde{\nu})|>\eta}+G\left(\log\frac{1}{2\delta}-\frac{\epsilon k}{2}\right)\right]\nonumber \\
        \label{eqn::gen_lem_app}
        &\geq  1-\Prr{E_1}- \Prrr{\sup_{\substack{\tilde{\nu}\in \am{\hat\nu}{3}\\x\in A_{\tau+\eta,\hat{\nu}}\cup A_{\tau+\eta,\tilde{\nu}}}}|O(x,\hat{\nu})-O(x,\tilde{\nu})|>\eta}-\delta e^{\epsilon},
\end{align}
where $$E_1=\left\{\inf_{y>0}e^{-\epsilon y/4\eta} \frac{\vol(A_{\tau+2\eta,\hat\nu}\cap A^c_{\tau-2\eta,\hat\nu})}{\vol(A_{\tau-y-2\eta,\hat{\nu}})}> e^{-\epsilon/2}\delta\right\}.$$

We now bound the deviation of the outlyingness function with high probability. 
To this end, note that Lemma~\ref{lem::ls_out} implies that for any $k\in[n]$,
\begin{align*}
        \Prrr{\sup_{\substack{\tilde{\nu}\in \am{\hat\nu}{k}\\x\in A_{\tau+\eta,\hat{\nu}}\cup A_{\tau+\eta,\tilde{\nu}}}}|O(x,\hat{\nu})-O(x,\tilde{\nu})|>\eta}\leq   \Prr{\Delta_{k}>\eta},
\end{align*} 
where $$\Delta_{k} = \frac{\tau  \tilde{S}_{n,k}(\sigma)+\tilde{S}_{n,k}(\mu)}{  \inf_{u\in\bS^{d-1}} b_{k,u}^-}.$$ 
Define $E_2 = \left\{\inf_{u\in \bS^{d-1}}b_{k,u}^-\geq c_1/2\right\}$. Then, conditioning on $E_2$ we have that \[\Delta_k \leq \frac{2}{c_1}\left[\tau\tilde{S}_{n,k}(\sigma)+\tilde{S}_{n,k}(\mu)\right],\] and so 
\begin{align*}\Prr{\Delta_k \geq \eta | E_2} &\leq \Prr{\frac{2}{c_1}\left[\tau\tilde{S}_{n,k}(\sigma)+\tilde{S}_{n,k}(\mu)\right] \geq \eta|E_2} \\ 
&\leq \Prr{\frac{2}{c_1}\left[\tau\tilde{S}_{n,k}\right] \geq \eta/2|E_2} + \Prr{\frac{2}{c_1}\left[\tilde{S}_{n,k}(\mu)\right] \geq \eta/2|E_2} \\
&= \Prr{\tilde{S}_{n,k} \geq c_1\eta/4\tau|E_2} + \Prr{\tilde{S}_{n,k}(\mu) \geq c_1\eta/4|E_2}.\end{align*}

Thus, applying Condition~\ref{cond::LS_Bound} gives \begin{align}
     \Prr{\Delta_{k}>\eta|E_2}&\lesssim  \Prr{\tilde{S}_{n,k}(\mu)>c_1\eta/4 |E_2}+ \Prr{ \tilde{S}_{n,k}(\sigma)>c_1 \eta/4\tau|E_2}\nonumber\\
     \label{eqn::delta_bound_general}
     &\lesssim a_{1}e^{-a_{2}c_1n \eta/4(\tau\vee 1)}+\kappa_{n}. 
\end{align}

Note $\Prr{E_2^c}$ is bounded based on \eqref{eq::E2C_bound}. Combining these past three results yields \begin{align*}
    &\Prrr{\sup_{\substack{\tilde{\nu}\in \am{\hat\nu}{k}\\x\in A_{\tau+\eta,\hat{\nu}}\cup A_{\tau+\eta,\tilde{\nu}}}}|O(x,\hat{\nu})-O(x,\tilde{\nu})|>\eta}\\
    &\qquad \leq \Prr{\Delta_k>\eta|E_2}\Prr{E_2} +  \Prr{\Delta_k>\eta|E_2^c}\Prr{E_2^c} \\
    &\qquad \leq \Prr{\Delta_k>\eta|E_2} + \Prr{E_2^c} \\
    &\qquad \lesssim a_{1}\exp\left(-a_{2}\frac{c_1n}{4}\left[1 \wedge \frac{\eta}{\tau\vee 1}\right]\right)+\kappa_{n}+\zeta_{n,\sigma}(\frac{c_1}{4}).
\end{align*}

We next bound $\Prr{E_1}$. To do so we will make use of Lemma~\ref{lem:vr_bound}. In order to apply the result, we require that $\inf_{\ubsd} \hat\sigma_u (\tau-y-2\eta)-\hat\alpha_{\mu}>0$. Note that, given $E_2$ and the restriction that $\tau \geq 4\eta + 16c_3/c_1$, \[\inf_{\ubsd} \hat\sigma_u (\tau-y-2\eta)-\hat\alpha_{\mu} \geq \frac{c_1(\tau - y - 2\eta)}{2} - \hat\alpha_{\mu} \geq c_1\eta + 8c_3 - \frac{c_1}{2}y - \hat\alpha_{\mu}.\] 
If we take $E_3 = \left\{\hat\alpha_{\mu}\leq 4c_3\right\}$, then further conditioning on $E_3$ gives that, for all $y \in \mathcal{L}$, where $\mathcal{L} = (0,\tau-2\eta - 8c_3/c_1)$, $\inf_{\ubsd} \hat\sigma_u (\tau-y-2\eta)-\hat\alpha_{\mu} > 0$. Thus, \begin{align*}
  \frac{\vol(A_{\tau+2\eta,\hat\nu}\cap A^c_{\tau-2\eta,\hat\nu})}{\vol(A_{\tau-y-2\eta,\hat{\nu}})} &\leq  \left(\frac{\sup_{\ubsd} \hat\sigma_u (\tau+2\eta)+\hat\alpha_{\mu}}{\inf_{\ubsd} \hat\sigma_u (\tau-y-2\eta) -\hat\alpha_{\mu}}\right)^d\\
  &\qquad -\left(\frac{\inf_{\ubsd} \hat\sigma_u (\tau-2\eta)-\hat\alpha_{\mu}}{\inf_{\ubsd} \hat\sigma_u (\tau-y-2\eta)-\hat\alpha_{\mu}}\right)^d \\  
    &\leq  \left(\frac{\sup_{\ubsd} \hat\sigma_u (\tau+2\eta)+\hat\alpha_{\mu}}{\inf_{\ubsd} \hat\sigma_u (\tau-y-2\eta)-\hat\alpha_{\mu}}\right)^d.
\end{align*}

Define $E_4 = \left\{\sup_{u\in \bS^{d-1}}\sigma(\hat\nu_{u})\leq 2c_2\right\}$, so that given $E_5 = E_2\cap E_3 \cap E_4$, \[\left(\frac{\sup_{\ubsd} \hat\sigma_u (\tau+2\eta)+\hat\alpha_{\mu}}{\inf_{\ubsd} \hat\sigma_u (\tau-y-2\eta)-\hat\alpha_{\mu}}\right)^d \leq \left(\frac{2c_2(\tau + 2\eta) + 4c_3}{c_1(\tau-y-2\eta)/2-4c_3}\right)^d.\] Thus, taken together, \begin{align*}
    \Prr{E_1} &\leq \Prr{E_1|E_5} + \Prr{E_5^c}\\
    &\leq \Prr{\inf_{y \in \mathcal{L}} e^{-\epsilon y/4\eta}\left(\frac{2c_2(\tau + 2\eta) + 4c_3}{c_1(\tau-y-2\eta)/2-4c_3}\right)^d > e^{-\epsilon/2}\delta} + \Prr{E_5^c}\\
    &\leq \ind{\inf_{y \in \mathcal{L}} e^{\epsilon/2-\epsilon y/4\eta}\left(\frac{2c_2(\tau + 2\eta) + 4c_3}{c_1(\tau-y-2\eta)/2-4c_3}\right)^d > \delta}+ \Prr{E_5^c}.
\end{align*}

Taking $y = \tau/2$, the inequality in the indicator is never satisfied for any \[\delta > e^{\epsilon/2-\epsilon y/4\eta}\left(\frac{2c_2(\tau + 2\eta) + 4c_3}{c_1(\tau-y-2\eta)/2-4c_3}\right)^d = \exp\left\{\frac{\epsilon}{2} - \frac{\epsilon\tau}{8\eta} + d\log\left(\frac{2c_2(\tau + 2\eta) + 4c_3}{c_1\tau/4-c_1\eta-4c_3}\right)\right\}.\] Thus, for any such $\delta$, we have $\Prr{E_1} \leq \Prr{E_5^c}$. 

Using equations~\eqref{eq::E2C_bound}, \eqref{eq::E3C_bound}, and \eqref{eq::E4C_bound} we have bounds on $\Prr{E_2^C}$, $\Prr{E_3^C}$, and $\Prr{E_4^C}$ respectively. These three bounds combine to give \[\Prr{E_5^c} \lesssim a_{1}e^{-a_{2}nc_1/4}+\kappa_{n}+\zeta_{n,\sigma}(\frac{c_1}{4})\wedge \zeta_{n,\mu}( c_3).\] Finally,
\begin{align*}
        1-\E{}{Z} &\leq  \Prr{E_1}+ \Prrr{\sup_{\substack{\tilde{\nu}\in \am{\hat\nu}{3}\\x\in A_{\tau+\eta,\hat{\nu}}\cup A_{\tau+\eta,\tilde{\nu}}}}|O(x,\hat{\nu})-O(x,\tilde{\nu})|>\eta}+e^{\epsilon}\delta \\
        &\lesssim \Prr{E_5^c}+a_{1}\exp\left(-a_{2}\frac{c_1n}{4}\left[1 \wedge \frac{\eta}{\tau\vee 1}\right]\right)+\kappa_{n}+\zeta_{n,\sigma}(\frac{c_1}{4})+e^{\epsilon}\delta \\
        &\lesssim  a_{1}e^{-a_{2}nc_1/4}+\kappa_{n}+\zeta_{n,\sigma}(\frac{c_1}{4})\wedge \zeta_{n,\mu}( c_3)
        \\
        &\qquad +a_{1}\exp\left(-a_{2}\frac{c_1n}{4}\left[1 \wedge \frac{\eta}{\tau\vee 1}\right]\right)+\kappa_{n}+\zeta_{n,\sigma}(\frac{c_1}{4})+e^{\epsilon}\delta \\
        &\lesssim \exp\left(-a_{2}\frac{c_1n}{4}\left[1 \wedge \frac{\eta}{\tau\vee 1}\right]\right) +\kappa_n + \zeta_{n,\sigma}(\frac{c_1}{4}) + e^{\epsilon}\delta \qedhere.
\end{align*}
\end{proof} 
The proof of Theorem~\ref{thm::pd_acc2} relies on the following concentration inequality. 
\begin{lemma}\label{lem::pd_acc2}
If Conditions~\ref{cond::Pop_Bound} and~\ref{cond::bounded-parameters} hold, then for all $\eta,\epsilon,\delta>0$, $t>0$, and $n,d\geq 1$ then
\begin{multline*}
\Pr\left(\norm{\tilde{\theta}-\theta}\geq t\big|\tilde{\theta}\neq \perp\right)\\ \lesssim \inf_{\substack{0<s<c_1\tau/2-c_3\\0<r<c_1\tau/2-s-c_3}} \Bigg[\exp\left(\epsilon\frac{(c_1/2+2c_2)(c_3+s)}{4\eta c_1c_2} -\frac{\epsilon t}{8\eta c_2} +\frac{\epsilon r}{2c_1\eta}+\log\frac{(2c_2\tau +s+c_3)^d-t^d}{r^d} \right)\\+\zeta_{n,\mu}\left(s \wedge c_3\right)+\zeta_{n,\sigma}\left( c_1/2\right)\Bigg].
\end{multline*}
\end{lemma}
\begin{proof}[Proof of Lemma~\ref{lem::pd_acc2}]
Let $\beta=\epsilon/4\eta$ and define the events $E_{6,s}=\{S_n(\mu)\leq s\}$ and $$E_7=\left\{\inf_{u\in \bS^{d-1}}\sigma(\hat\nu_{u})\geq c_1/2\right\}\cap \left\{\sup_{u\in \bS^{d-1}}\sigma(\hat\nu_{u})\leq 2c_2\right\}\cap \left\{\hat\alpha_{\mu}\leq 4c_3\right\}.$$
Condition~\ref{cond::Pop_Bound} directly gives that $\Prr{E_{6,s}^c} \lesssim \zeta_{n,\mu}(s)$. Using this, the identities established in \eqref{eq::E3C_bound}, \eqref{eq::E4C_bound}, and \eqref{eq::E7C_bound}, and noting that $c_1 \leq c_2$, yields that for all $s\geq 0$, we have
\begin{equation}\label{eqn::accu_step_1}
    \Prr{(E_7\cap E_{6,s})^c} \lesssim \zeta_{n,\mu}(s\wedge c_3)+ \zeta_{n,\mu}(c_1/2).
\end{equation}

Next, condition on $E_7\cap E_{6,s}$.
For any $\norm{x-\theta}\geq t$,
\begin{align}
   O(x,\hat{\nu}) &= \sup_{\ubsd} \frac{x^\top u - \theta^\top u + \theta^\top u - \hat{\mu}_{u}}{\sigma(\hat\nu_{u})} \nonumber\\ 
   &\geq \frac{\sup_{\ubsd} (x - \theta)^\top u -\sup_{\ubsd}|\hat\mu_u-\theta^\top u|}{\sup_{\ubsd}\hat\sigma_u}\nonumber\\
    &= \frac{\norm{x-\theta} -\sup_{\ubsd}|\hat\mu_u-\mu_u + \mu_u - \theta^\top u|}{\sup_{\ubsd}\hat\sigma_u} \nonumber\\
    &\geq \frac{\norm{x-\theta} - \sup_{\ubsd}|\hat\mu_u-\mu_u| - \sup_{\ubsd}|\mu_u - \theta^\top u|}{\sup_{\ubsd}\hat\sigma_u}\nonumber\\
    &\geq \frac{t-c_3-s}{2c_2}.\label{eq::O_lowerbound_ball}
\end{align} The last inequality relies on $E_{6,s}$ and Condition~\ref{cond::bounded-parameters}.
A similar argument shows that, for any $r>0$ and $x\in\rdd$ such that $\norm{x-\theta}\leq r$,
\begin{align}
   O(x,\hat{\nu}) &= \sup_{\ubsd} \frac{x^\top u - \theta^\top u + \theta^\top u - \hat{\mu}_{u}}{\sigma(\hat\nu_{u})} \nonumber\\ 
   &\leq \frac{\sup_{\ubsd}(x^\top u - \theta^\top u) + \sup_{\ubsd} (\theta^\top u - \hat{\mu}_{u})}{\inf_{\ubsd} \sigma(\hat\nu_{u})} \nonumber \\
   &\leq \frac{r +\sup_{\ubsd}|\theta^\top u-\mu_u|+\sup_{\ubsd}|\mu_u-\hat\mu_u|}{c_1/2} \nonumber \\ 
   &\leq \frac{r+c_3+s}{c_1/2}.\label{eq::O_upperbound_ball}
\end{align}
Therefore, taking $B_{r}(x)$ to be the ball of radius $r$ centered on $x$, then
\begin{equation}\label{eqn::in_ball}
B_{c_1\tau/2 -s-c_3}(\theta) \subset A_{\tau,\hat\nu} = \{x \mid O(x,\hat\nu) < \tau\}\subset B_{2c_2\tau +s+c_3}(\theta).
\end{equation} Here, the first inclusion follows from \eqref{eq::O_upperbound_ball} and the second from \eqref{eq::O_lowerbound_ball}.

Next, note that by definition \[\Prr{\norm{\tilde{\theta}-\theta}\geq t|E_7\cap E_{6,s}} = \frac{ \int_{B_t^c(\theta)\cap A_{\tau,\hat\nu}} \exp{\left(-\beta O(x,\hat{\nu})\right)}dx}{ \int_{A_{\tau,\hat\nu}} \exp{\left(-\beta O(x,\hat{\nu})\right)}dx}.\] 
Moreover, if $t \geq 2c_2\tau +s+c_3$, \eqref{eqn::in_ball} indicates that $B_{t}^c(\theta)\cap A_{\tau,\hat\nu} = \emptyset$ and so \[\int_{B_t^c(\theta)\cap A_{\tau,\hat\nu}} \exp{\left(-\beta O(x,\hat{\nu})\right)}dx=0.\] 
As such, we restrict $t< 2c_2\tau +s+c_3$. Under this restriction, applying \eqref{eq::O_lowerbound_ball} we find that \begin{align*}
    \int_{B_t^c(\theta)\cap A_{\tau,\hat\nu}} \exp{\left(-\beta O(x,\hat{\nu})\right)}dx &\leq \int_{t \leq \norm{x-\theta} \leq 2c_2\tau + s + c_3} \exp{\left(-\beta O(x,\hat{\nu})\right)}dx \\
    &\leq \int_{t \leq \norm{x-\theta} \leq 2c_2\tau + s + c_3} \exp{\left(\frac{-\beta(t - c_3 - s)}{2c_2}\right)}dx \\
    &= \exp{\left(\frac{-\beta(t - c_3 - s)}{2c_2}\right)}\vol(\{t\leq \norm{x}\leq 2c_2\tau +s+c_3\}).
\end{align*}

Similarly, an application of \eqref{eq::O_upperbound_ball} gives
\begin{align*}
    \int_{A_{\tau,\hat\nu}}\exp{\left(-\beta O(x,\hat{\nu})\right)}dx &\geq \sup_{0 < r \leq c_1\tau/2 - s - c_3} \int_{B_r(\theta)} \exp{\left(-\beta O(x,\hat{\nu})\right)}dx \\
    &\geq \sup_{0 < r \leq c_1\tau/2 - s - c_3} \int_{B_r(\theta)} \exp{\left(\frac{-\beta(r + c_3 + s)}{c_1/2}\right)}dx \\
    &\geq \sup_{0 < r \leq c_1\tau/2 - s - c_3} \exp{\left(\frac{-\beta(r + c_3 + s)}{c_1/2}\right)}\vol(B_r(\theta)).
\end{align*}

Combining these two bounds gives \begin{align*}
    &\Prr{\norm{\tilde{\theta}-\theta}\geq t|E_7\cap E_{6,s}} \\
    &\leq \frac{\exp{\left(\frac{-\beta(t - c_3 - s)}{2c_2}\right)}\vol(\{t\leq \norm{x}\leq 2c_2\tau +s+c_3\})}{\sup_{0 < r \leq c_1\tau/2 - s - c_3} \exp{\left(\frac{-\beta(r + c_3 + s)}{c_1/2}\right)}\vol(B_r(\theta))} \\
    &= \inf_{0 < r \leq c_1\tau/2 - s - c_3} \frac{\exp{\left(\frac{-\beta(t - c_3 - s)}{2c_2}\right)}\vol(B_t^c(0)\cap B_{2c_2\tau + s + c_3}(0))}{\exp{\left(\frac{-\beta(r + c_3 + s)}{c_1/2}\right)}\vol(B_r(\theta))} \\
    &\lesssim \inf_{0 < r \leq c_1\tau/2 - s - c_3} \exp{\left(-\beta\frac{c_1/2(t - c_3 - s)-2c_2(r + c_3 + s)}{c_2c_1}\right)}\frac{(2c_2\tau + s + c_3)^d - t^d}{r^d} \\
    &= \inf_{0 < r \leq c_1\tau/2 - s - c_3} \exp{\left(\beta\left(\frac{(c_1/2 + 2c_2)(c_3 + s)}{c_1c_2}-\frac{t}{2c_2} + \frac{2 r}{c_1}\right)\right)}\frac{(2c_2\tau + s + c_3)^d - t^d}{r^d}.
\end{align*}
This bound holds for all $0 < s < c_1\tau/2 - c_3$, and so we can take the least upper bound on the range. Combining this and re-writing the expression gives \begin{multline*}\Prr{\norm{\tilde{\theta}-\theta}\geq t|E_7\cap E_{6,s}} \lesssim \inf_{\substack{0 < r \leq c_1\tau/2 - s - c_3 \\ 0 < s < c_1\tau/2 - c_3}} \exp\left(\beta\left(\frac{(c_1/2 + 2c_2)(c_3 + s)}{c_1c_2}-\frac{t}{2c_2} + \frac{2 r}{c_1}\right) \right.\\ 
\left.+ \log\left(\frac{(2c_2\tau + s + c_3)^d - t^d}{r^d}\right)\right).\end{multline*}

Since $\Prr{\norm{\tilde{\theta}-\theta}\geq t} \leq \Prr{\norm{\tilde{\theta}-\theta}\geq t|E_7\cap E_{6,s}} + P((E_7\cap E_{6,s})^c)$, combining this result with \eqref{eqn::accu_step_1} gives the desired result.
\end{proof}
\noindent We now prove Theorem~\ref{thm::pd_acc2}. 
\begin{proof}[Proof of Theorem~\ref{thm::pd_acc2}]
To prove Theorem~\ref{thm::pd_acc2}, it suffices to find a lower bound on $t$ such that  
\begin{equation*}
     \Prr{\norm{\tilde{\theta}-\theta}\geq t}\leq \Prr{\norm{\tilde{\theta}-\theta}\geq t\big|\tilde{\theta}\neq \perp}+\Prr{\tilde{\theta}= \perp} \leq \gamma.
\end{equation*}
Furthermore, by assumption, $\Prr{\tilde{\theta}= \perp}=1-\E{}{Z}\leq \gamma/3$. 
Thus, it suffices to find $t$ such that 
$$\Prr{\norm{\tilde{\theta}-\theta}\geq t\big|\tilde{\theta}\neq \perp}\leq 2\gamma/3.$$
To this end, applying Lemma~\ref{lem::pd_acc2}, yields
\begin{multline}\label{eqn::ci}    
\Pr\left(\norm{\tilde{\theta}-\theta}\geq t\big|\tilde{\theta}\neq \perp\right)\\ \lesssim \inf_{\substack{0<s<c_1\tau/2-c_3\\0<r<c_1\tau/2-s-c_3}} \Bigg[\exp\left(\frac{\epsilon(c_1/2+2c_2)(c_3+s)}{4\eta c_1c_2} -\frac{\epsilon t}{8\eta c_2} +\frac{\epsilon r}{2c_1\eta}+\log\frac{(2c_2\tau +s+c_3)^d-t^d}{r^d} \right)\\+\zeta_{n,\mu}\left(s \wedge c_3\right)+\zeta_{n,\sigma}\left( c_1/2\right)\Bigg].
\end{multline}
We first show that 
{\small
\begin{equation}\label{eqn::gamma_bound}    
\inf_{\substack{0<s<\frac{c_1\tau}{2}-c_3\\0<r<\frac{c_1\tau}{2}-s-c_3}} \Bigg[\exp\left(\frac{\epsilon(\frac{c_1}{2}+2c_2)(c_3+s)}{4\eta c_1c_2} -\frac{\epsilon t}{8\eta c_2} +\frac{\epsilon r}{2c_1\eta}+\log\frac{(2c_2\tau +s+c_3)^d-t^d}{r^d} \right)\Bigg]\leq \frac{\gamma}{3}.
\end{equation}}
To this end, using the fact that $c_1<c_2$, for any $s,r>0$, we have that
\begin{align*}
\frac{(c_1/2+2c_2)(c_3+s)}{4 c_1c_2} -\frac{ t}{8 c_2} +\frac{ r}{2c_1}  &\leq \frac{5(c_3 + s)}{8c_1} -\frac{t}{8c_2} +\frac{r}{2c_1}= \frac{5(c_3+s)+4r-tc_1/c_2}{8c_1}.
\end{align*}
Next, for some $k_r,k_s>0$, take $r = tc_1/(k_rc_2)$ and $s = r/k_s $. Then
\begin{align*}
\frac{5(c_3+s)+4r-tc_1/c_2}{8c_1} &= \frac{(5+4k_s)tc_1/(k_rc_2k_s)-tc_1/c_2+5c_3}{8c_1}= \frac{(5+4k_s-k_rk_s)t}{8c_2k_rk_s}+\frac{5c_3}{8c_1}.
\end{align*}

In order to apply this result to \eqref{eqn::gamma_bound}, it must be the case that $0 < s < c_1\tau/2 - c_3$ and $0 < r < c_1\tau/2 - s - c_3$. 
Note that $ r < c_1\tau/2 - s - c_3\implies  r+s < c_1\tau/2  - c_3\implies s < c_1\tau/2 - c_3$, so it suffices to consider $s>0$,  and $0 < r < c_1\tau/2 - s - c_3$. 
First, observe that $s,r>0$ always holds. 
Next, simplifying the second bound yields
\begin{align*}
     r < c_1\tau/2 - s - c_3&\implies  r+s < c_1\tau/2  - c_3\\
     &\implies  \frac{tc_1}{k_rc_2}(1+k_s^{-1}) < c_1\tau/2-c_3 \\ &\implies  t< \frac{k_rc_2\tau-2k_rc_2c_3/c_1}{2(1+k_s^{-1})} .
\end{align*}
Assuming that this bound on $t$ holds, then 

\begin{multline}\label{eqn::prelim_bound}
    \exp\left(\epsilon\frac{(c_1/2+2c_2)(c_3+s)}{4\eta c_1c_2} -\frac{\epsilon t}{8\eta c_2} +\frac{\epsilon r}{2c_1\eta}+\log\frac{(2c_2\tau +s+c_3)^d-t^d}{r^d} \right)\\
    \leq  \exp\left(\left[\frac{5+4k_s-k_rk_s}{8k_rk_s}\right]\frac{t\epsilon}{c_2\eta}+\log\frac{(2c_2\tau +s+c_3)^d-t^d}{r^d}+\frac{5c_3}{8c_1} \right). 
\end{multline}
We next simplify the logarithmic term. 
First suppose that $r=tc_1/(k_rc_2)$, then,
\begin{align}
  \log\frac{(2c_2\tau +s+c_3)^d-t^d}{r^d} &=   \log\frac{(2c_2\tau +tc_1/k_rc_2+c_3)^d-t^d}{(tc_1/k_rc_2)^d}\nonumber\\
  &=d\log \frac{k_rc_2}{c_1} +\log\frac{(2c_2\tau +tc_1/k_rc_2+c_3)^d-t^d}{t^d}\nonumber\\
  &\leq d\log \frac{k_rc_2}{c_1} +d\log\frac{2c_2\tau +tc_1/k_rc_2+c_3}{t}\nonumber\\
  &= d\log \frac{k_rc_2}{c_1} +d\log\left(\frac{2c_2\tau+c_3} {t}+\frac{c_1}{k_rc_2}\right)\nonumber\\
  \label{eqn::first_bound}
  &= d\log\left(\frac{2k_rc_2^2\tau+c_3k_rc_2} {tc_1}+1\right).
\end{align}
Now, from Lemma~\ref{lem::pd_acc2}, if $t\geq 2c_2\tau+s+c_3$, then $\Pr\left(\norm{\tilde{\theta}-\theta}\geq t\big|\tilde{\theta}\neq \perp\right)=0$ and so we consider $t < 2c_2\tau + s + c_3$. 
Based on the assumed values for $r$ and $s$ this gives 
\[ t < 2c_2\tau + \frac{tc_1}{k_sk_rc_2}+c_3 \implies t < \frac{2k_sk_r\tau c_2^2+c_3k_sk_r c_2}{k_rk_sc_2 - c_1},\] 
whenever $k_rk_sc_2 - c_1 > 0$. This holds trivially for any $k_rk_s \geq 1$. In order to ensure that the bounds discussed holds, we require that this upper bound on $t$ is less than the previous upper bound on $t$ to account for all values of $t$. Thus, using the fact that $k_rk_s\geq 1$, we have that \begin{align*}
 \frac{2k_sk_r\tau c_2^2+c_3k_sk_r c_2}{k_rk_sc_2 - c_1} \leq  \frac{k_rc_2\tau-2k_rc_2c_3/c_1}{2(1+k_s^{-1})},
\end{align*} 
whenever,
\begin{align*}
 &\frac{2k_sk_r\tau c_2^2+c_3k_sk_r c_2}{k_rk_sc_2 - c_1} \leq  \frac{k_rc_2\tau-2k_rc_2c_3/c_1}{2(1+k_s^{-1})}\\
 &\implies \frac{2k_s\tau c_2+c_3k_s}{k_rk_sc_2 - c_1} \leq  \frac{\tau-2c_3/c_1}{2(1+k_s^{-1})}\\
 &\implies \frac{2(k_s+1)\tau c_2+c_3(k_s+1)}{k_rk_sc_2 - c_1}+\frac{c_3}{c_1} \leq  \tau/2\\
  &\implies \frac{c_3(k_s+1)}{k_rk_sc_2 - c_1}+\frac{c_3}{c_1} \leq  \tau \frac{k_rk_sc_2 - c_1-4(k_s+1) c_2}{2(k_rk_sc_2 - c_1)}\\
  &\implies \tau\geq \frac{2(k_rk_sc_2 - c_1)}{k_rk_sc_2 - c_1-4(k_s+1) c_2}\left(\frac{c_3(k_s+1)}{k_rk_sc_2 - c_1}+\frac{c_3}{c_1} \right).
\end{align*} 
Now, say we take $k_s=1$ and $k_r=10$, then 
\begin{align*}
\tau\geq \frac{20 -2(c_1/c_2)}{2 - (c_1/c_2)}\left(\frac{2c_3}{10c_2 - c_1}+\frac{c_3}{c_1} \right),
\end{align*}
for which we can simplify this to the requirement that
\begin{align*}
\tau\geq 20\left(\frac{2c_3}{10c_2 - c_1}+\frac{c_3}{c_1} \right),
\end{align*}

Thus, for this choice of $r$ and $s$, \eqref{eqn::prelim_bound} becomes \begin{multline*}
\exp\left(\frac{\epsilon(c_1/2+2c_2)(c_3+s)}{4\eta c_1c_2} -\frac{\epsilon t}{8\eta c_2} +\frac{\epsilon r}{2c_1\eta}+\log\frac{(2c_2\tau +s+c_3)^d-t^d}{r^d}+\frac{5c_3}{8c_1}\right)\\
    \leq  \exp\left(\left[\frac{5+4k_s-k_rk_s}{8k_rk_s}\right]\frac{t\epsilon}{c_2\eta}+d\log\left(\frac{2k_rc_2^2\tau} {tc_1}+1\right)+\frac{5c_3}{8c_1}\right). 
\end{multline*} As a result, the above argumentation gives consideration of \begin{equation}
   t < \frac{k_rc_2\tau-2k_rc_2c_3/c_1}{2(1+k_s^{-1})}\label{eq::t_bounds}.
\end{equation}

Using the derived upper bound, for $t>1/n$, 
\begin{multline*}
    \exp\left(\left[\frac{5+4k_s-k_rk_s}{8k_rk_s}\right]\frac{t\epsilon}{c_2\eta}+d\log\left(\frac{2k_rc_2^2\tau} {tc_1}+1\right)+\frac{5c_3}{8c_1}\right) \\
    \leq \exp\left(\left[\frac{5+4k_s-k_rk_s}{8k_rk_s}\right]\frac{t\epsilon}{c_2\eta}+d\log\left(\frac{2nk_rc_2^2\tau} {c_1}+1\right)+\frac{5c_3}{8c_1}\right).
\end{multline*} Bounding this above by $\gamma/3$ yields
\begin{align*}
    &\exp\left(\left[\frac{5+4k_s-k_rk_s}{8k_rk_s}\right]\frac{t\epsilon}{c_2\eta}+d\log\left(\frac{2nk_rc_2^2\tau} {c_1}+1\right)\right) \leq \frac{\gamma}{3}\exp\left(-\frac{5c_3}{8c_1}\right) \\
    &\implies \left[\frac{k_rk_s-5-4k_s}{8k_rk_s}\right]\frac{t\epsilon}{c_2\eta} \geq -\log\left(\frac{\gamma}{3}\exp\left(-\frac{5c_3}{8c_1}\right)\left(\frac{2nk_rc_2^2\tau} {c_1}+1\right)^{-d}\right) \\
    &\implies t > \frac{\eta c_2}{\epsilon}\left[\frac{8k_rk_s}{k_rk_s-5-4k_s}\right]\log\left(\frac{3}{\gamma}\exp\left(\frac{5c_3}{8c_1}\right)\left(\frac{2nk_rc_2^2\tau} {c_1}+1\right)^{d}\right)\\
    &\implies t > 80\frac{\eta c_2}{\epsilon}\log\left(\frac{3}{\gamma}\exp\left(\frac{5c_3}{8c_1}\right)\left(\frac{20nc_2^2\tau} {c_1}+1\right)^{d}\right)
\end{align*} 
Therefore, this reduces to
\begin{align*}
   t \gtrsim \frac{\eta c_2}{\epsilon}\left[\log(1/\gamma)\vee d\log\left(\frac{20c_2^2n\tau}{c_1}+1 \right)\vee c_3/c_1\right] \vee 1/n.
\end{align*} 

Next we bound $\zeta_{n,\mu}(s\wedge c_3)+\zeta_{n,\sigma}(c_1/2) \leq \gamma/3$. By assumption, $\gamma > 3(zeta_{n,\sigma}(\frac{c_1}{2})\vee \zeta_{n,\mu}(c_3))$, and so we consider only $\zeta_{n,\mu}(s)$ for $s <  c_3$. To this end, using the fact that $\zeta_{n,\mu}$ are non-increasing, we have that $$\zeta_{n,\mu}(s)\leq \gamma/3\implies t\gtrsim \frac{c_2}{c_1}\zeta_{n,\mu}^{-1}(\gamma/3).$$
This completes the proof. 
\end{proof}
\section{Proofs from Section~\ref{sec::examples}}\label{app::ex_pro}
Before proving Theorems~\ref{thm::trimmed} and~\ref{thm::order_cond_1}, we first prove two results concerning the projected quantiles and spacings. 
Let $\lambda_{t,q_1,q_2}=\inf_{\substack{u\in\bS^{d-1}\\z\in(\xi_{q_1,u}-t,\xi_{q_2,u}+t)}}f_{\nu_u}(z)$ and $\hat\xi_{q,u}=F^{-1}_{\hat\nu_u}(q)$.
\begin{lemma}\label{lem::quantile_process_concen}
If $\nu$ is absolutely continuous, then for all for $0<q_1\leq q_2<1$, $n,d\geq 1$ and $t>0$ it holds that
$$\Prr{\sup_{u\in\bS^{d-1}}\sup_{x\in [q_1,q_2]}|\xi_{x,u}-\hat\xi_{x,u}|\geq t}\leq \left(\frac{cn}{d}\right)^de^{-2nt^2\lambda_{t,q_1,q_2}^2}.$$
\end{lemma}
\begin{proof}
Note that
\begin{equation}\label{eqn::start}
\Prrr{\sup_{\substack{x\in [q_1,q_2]\\ u\in\bS^{d-1}}}|\xi_{x,u}-\hat\xi_{x,u}|\geq t}=\Prrr{\bigcup_{\substack{x\in [q_1,q_2]\\ u\in\bS^{d-1}}}\left(\{\hat\xi_{x,u}-\xi_{x,u}\geq t\}\cup \{\xi_{x,u}-\hat\xi_{x,u}\geq t\}\right) }.
\end{equation}
Furthermore, a direct consequence of the mean value theorem is that
\begin{align*}
    \{\hat\xi_{x,u}-\xi_{x,u}\geq t\}&= \{x\geq F_{\hat\nu_u}(\xi_{x,u}+t)\}\\
    &= \{ F_{\nu_u}(\xi_{x,u}+t)-F_{\hat\nu_u}(\xi_{x,u}+t)\geq F_{\nu_u}(\xi_{x,u}+t)-x\}\\
    &\subset\{ \sup_{\substack{y\in[0,1]\\ u\in\bS^{d-1}}}|F_{\nu_u}(y)-F_{\hat\nu_u}(y)|\geq t\inf_{\substack{z\in[q_1,q_2]\\ u\in\bS^{d-1}}}f_{\nu_u}(\xi_{z,u}+t)\}. 
\end{align*}
A similar argument yields 
\begin{align*}
    \{\xi_{x,u}-\hat\xi_{x,u}\geq t\}&\subset\{ \sup_{\substack{y\in[0,1]\\ u\in\bS^{d-1}}}|F_{\nu_u}(y)-F_{\hat\nu_u}(y)|\geq t\inf_{\substack{z\in[q_1,q_2]\\ u\in\bS^{d-1}}}f_{\nu_u}(\xi_{z,u}-t)\}. 
\end{align*}
Combining these two set inequalities with \eqref{eqn::start} and applying Talagrand's inequality \citep{Talagrand1994} results in 
\begin{align*}
    \Prrr{\sup_{\substack{x\in [q_1,q_2]\\ u\in\bS^{d-1}}}|\xi_{x,u}-\hat\xi_{x,u}|\geq t}&\leq \Prrr{\sup_{\substack{y\in[0,1]\\ u\in\bS^{d-1}}}|F_{\nu_u}(y)-F_{\hat\nu_u}(y)|\geq t\inf_{\substack{u\in\bS^{d-1}\\z\in(\xi_{q_1,u}-t,\xi_{q_2,u}+t)}}f_{\nu_u}(z)}\\
    &=\Prrr{\sup_{\substack{y\in[0,1]\\ u\in\bS^{d-1}}}|F_{\nu_u}(y)-F_{\hat\nu_u}(y)|\geq t\lambda_{t,q_1,q_2}}\\
    &\leq \left(\frac{cn}{d}\right)^de^{-2nt^2\lambda_{t,q_1,q_2}^2}. \qedhere
\end{align*}
\end{proof}
\noindent For $\sam{X}{n}\in \bD_{n\times d}$ and $\ubsd$, let $(X^\top u)_{(k)}$ denote the $k$th order statistic of the sample $X_1^\top u, \ldots, X_n^\top u$. 
\begin{lemma}\label{thm::order_spacing_new_uniform}
If, for a given $0<q\leq 1/2$, Condition~\ref{cond::dens_lb} holds for $(q,\nu,k)$, then there exists a universal constant $c>0$ such that for all $d,n\geq 1$ and all $C>0$, it holds that
\begin{multline*}
      \Prrr{\sup_{\substack{u\in\bS^{d-1}\\ p\in [q,1-q-k/n]}}(X^\top u)_{(\floor{pn}+k)}-(X^\top u)_{(\floor{pn})}\geq C/n}\lesssim  \exp\left(-nr^2M^2+d\log(cn/d)\right)\\
      +\exp\left(-\frac{3MC}{4k}+d\log 9
      +\log n(1-2q+1/n)\right).
\end{multline*}
\end{lemma}
\begin{proof}
Let $m=\floor{qn}$.
Now, define
$$E_8=\bigcap_{\ubsd}\left\{\xi_{q,u}-r\leq (X^\top u)_{(m)}\leq (X^\top u)_{(m+k)}\leq \xi_{q,u}+r\right\}.$$
Now, noting that, for a given $j\in[n]$, $(X^\top u)_{(j)}=\hat\xi_{j/n,u}$, applying Lemma~\ref{lem::quantile_process_concen} with $q_1=q$, $q_2=q+k/n$ and $t=r$ in conjunction with
Condition~\ref{cond::dens_lb} yields
\begin{equation}\label{eqn::spac_step_1}
\Prr{E_8^c}\lesssim \exp\left(-nr^2\lambda_{r,q,q+k/n}^2+d\log(cn/d)\right)\leq \exp\left(-nr^2M^2+d\log(cn/d)\right).
\end{equation}
Now, let
$$E_9=\left\{\max_{m\leq j<n-m}\sup_{\ubsd}\left[(X^\top u)_{(j+1)}-(X^\top u)_{(j)}\right]\geq \frac{C}{kn}\right\},$$
and note that
\begin{equation}\label{eqn::spac_step_2}
        \Prrr{\ \sup_{\substack{u\in\bS^{d-1}\\ p\in [q,1-q-k/n]}}(X^\top u)_{(\floor{pn}+k)}-(X^\top u)_{(\floor{pn})}\geq \frac{C}{n}}\leq \Prr{E_9|E_8}+\Prr{E_8^c}.
\end{equation}
We now consider the event $E_9$. 
Recall that the set $\cN_\epsilon$ is an $\epsilon$-net on $\bS^{d-1}$ if for all $\ubsd$, there exists $v\in \cN_\epsilon$ such that $\norm{u-v}\leq \epsilon$. 
Letting $\cN_\epsilon$ be an arbitrary $\epsilon$-net on $\bS^{d-1}$, it is straightforward (see, e.g., \citep[][Lemma 4.4.1]{vershynin2018high}) to show that for any $i,j\in[n]$,
$$\sup_{\ubsd}|X_i^\top u-X_j^\top u|\leq \frac{1}{1-\epsilon}\sup_{u\in \cN_\epsilon}|X_{i}^\top u-X_{j}^\top u|.$$
Furthermore, there exists $\cN_{1/4}$ such that $|\cN_{1/4}|\leq 9^d$ \citep[][see Corollary 4.2.13]{vershynin2018high}. 
Applying these two facts, in conjunction with two applications of subadditivity of probability measures, yields
\begin{align}
    \Prr{E_8} 
    &\leq \Prr{\max_{m\leq j<n-m}\sup_{u\in \cN_{1/4}}(X^\top u)_{(j+1)}-(X^\top u)_{(j)}\geq \frac{3C}{4n}}\nonumber \\
    &\leq 9^d\sup_{u\in \cN_{1/4}}\Prr{\max_{m\leq j<n-m}(X^\top u)_{(j+1)}-(X^\top u)_{(j)}\geq \frac{3C}{4n}}\nonumber \\
    &\leq 9^d(n-2m+1)\sup_{\substack{u\in \cN_{1/4} \\ m\leq j<n-m}}\Prr{(X^\top u)_{(j+1)}-(X^\top u)_{(j)}\geq \frac{3C}{4n}}\nonumber \\
    \label{eqn::spac_step_3}
    &\coloneqq e^{d\log 9+\log(n-2m+1)}\sup_{\substack{u\in \cN_{1/4}\\ m\leq j<n-m}}\Prr{A_{3,u,j}}.
\end{align}
Now, observe that for any $m\leq j<n-m$ and $\ubsd$, using Condition~\ref{cond::dens_lb}, 
together with the mean value theorem, imply that on $E_8$ and $A_{3,u,j}$ it holds that
\begin{align}\label{eqn::intermediate_spacing_step}
    &\frac{F_{\nu_u}((X^\top u)_{(j+1)})-F_{\nu_u}((X^\top u)_{(j)})}{(X^\top u)_{(j+1)}-(X^\top u)_{(j)}}\geq M\implies F_{\nu_u}((X^\top u)_{(j+1)})-F_{\nu_u}((X^\top u)_{(j)}) \geq   \frac{3C}{4n}M.
\end{align}
Let $U_{(j),u}=F_{\nu_u}((X^\top u)_{(j)})$. 
Now, recall that for any set of $n$ independent standard uniform random variables $V_1,\ldots,V_n$, for any $j\in[n-1]$, it holds that $V_{(j+1)}-V_{(j)}\eqd V_{(1)}$ and $V_{(1)}\sim Beta(1,n)$ \citep[see, e.g.,][]{Pyke1965}. 
These two facts, combined inequality \eqref{eqn::intermediate_spacing_step}, results in 
\begin{align}
        \sup_{\substack{u\in \cN_{1/4}\\ m\leq j<n-m}}\Prr{A_{3,u}|E_8}&\leq\sup_{\substack{u\in \cN_{1/4}\\ m\leq j<n-m}}\Prr{U_{(j+1),u}-U_{(j),u}\geq 3CM/4n}\nonumber\\
         &\leq\sup_{u\in \cN_{1/4}}\Prr{U_{(1),u}\geq 3CM/4n}\nonumber\\
          &\leq (1-(3CM/4n)\wedge 1)^{n}\nonumber\\
          \label{eqn::spac_step_4}
          &\leq e^{-3MC/4k}.
\end{align}
Combining \eqref{eqn::spac_step_1}--\eqref{eqn::spac_step_4} yields 
\begin{multline*}
     \Prrr{\sup_{\substack{u\in\bS^{d-1}\\ p\in [q,1-q-k]}}(X^\top u)_{(\floor{pn}+k)}-(X^\top u)_{(\floor{pn})}\geq \frac{C}{n}}\lesssim \exp\left(-nr^2M^2+d\log(cn/d)\right)\\
     \exp\left(-3MC/4k+d\log 9+\log(n-2m+1)\right).\qedhere
\end{multline*}
\end{proof}
\begin{proof}[Proof of Theorem~\ref{thm::trimmed}]
Note that:
\begin{align*}
  \sup_{\tilde\nu\in \tilde\cM_1(\hat\nu)}S_{n,k}(\mu)\leq \frac{\max_{0\leq i\leq k+1}X_{(n-\alpha n +1+k-i)}-X_{(\alpha n +1-i)}}{n-2\alpha n}.
\end{align*}
Therefore, applying Lemma~\ref{thm::order_spacing_new_uniform} we have that 
\begin{align*}
   \Prr{S_{n,3}(\mu)\geq \eta}&\lesssim  \max_{0\leq i\leq 4}\Prr{\sup_{u}\frac{(X^\top u)_{(n-\alpha n +4-i)}-(X^\top u)_{(\alpha n +1-i)}}{n-2\alpha n}\geq  \eta}\\
   &\lesssim e^{-Mn^2(1-2\alpha) \eta/4+d+\log(n(1-2\alpha))}+e^{-nM^2r^2+d\log(cn/d)}. 
\end{align*}
This proves the assertion for the trimmed mean. We now focus on the trimmed absolute deviations. 
Consider some $\tilde\nu\in \tilde\cM_1(\hat\nu,3)$ and let $Y_1,\ldots,Y_n$ be the observations that correspond to $\tilde\nu$. 
For a particular $\ubsd$, let $A_{1,u},A_{2,u}\subset[n]$ be the subset of indices which are not trimmed from $\hat\nu_u$ and $\tilde\nu_u$, respectively. 
Now, first, it follows from above that with probability $1-e^{-Mn(1-2\alpha)\log n/4+d+\log(n(1-2\alpha))}-e^{-nM^2r^2+d\log(cn/d)}$, we have that $S_{n,3}(\mu)=\sup_{u,\tilde\nu}|\hat\mu_u-\tilde\mu_u|\leq \eta$. 
Therefore, we have that for all $\ubsd$ and $i\in[n]$, it holds that
$$|Y_i^\top u-\tilde\mu_u|\leq \begin{cases}
    sign(Y_i^\top u-\hat\mu_u)(Y_i^\top u-\hat\mu_u+\hat\mu_u-\tilde\mu_u)& |Y_i^\top u-\hat\mu_u|\geq \eta\\
    |Y_i^\top u-\hat\mu_u+\hat\mu_u-\tilde\mu_u| & |Y_i^\top u-\hat\mu_u|\leq \eta 
    \end{cases}.$$
Using this representation, we have that for $i\in A_{1,u}\cap A_{2,u}$, we have that$||X_i^\top u-\hat\mu_u|-|Y_i^\top u-\tilde\mu_u||\leq 2\eta.$
Applying this inequality yields that
\begin{align*}
   S_{n,3}(\sigma)&=\sup_{\ubsd}\left|\sum_{i\in A_{1,u}} \frac{|X_i^\top u-\hat\mu_u|}{n(1-2\alpha)}-\sum_{i\in A_{2,u}} \frac{|Y_i^\top u-\tilde\mu_u|}{n(1-2\alpha)}\right|\\
   &\leq 3\sup_{\ubsd}\sup_{i\in [3]}\left|\frac{|X_i^\top u-\hat\mu_u|-|Y_i^\top u-\hat\mu_u|}{n(1-2\alpha)}\right|+2\eta\\
   &\lesssim \sup_{\ubsd}\sup_{i\in [3]}\left|\frac{X_i^\top u-Y_i^\top u}{n(1-2\alpha)}\right|+\eta\\
   &\lesssim \eta,
\end{align*}
where the last inequality comes from the fact that $S_{n,3}(\mu)\leq \eta$. 
It follows that there exists a universal constant $C>0$ such that 
\begin{align*}
   \Prr{S_{n,3}(\sigma)\geq \eta}\leq \Prr{S_{n,3}(\sigma)\geq \eta/C},
\end{align*}
and this completes the proof. 
\end{proof}
The proof of Theorem~\ref{thm::order_cond_1} is a consequence of two lemmas, which we present as follows. 
\begin{lemma}\label{lem::med_cond_1}
If Condition~\ref{cond::dens_lb} holds for $(1/2,\nu,k)$, then there exists a universal constant $c>0$ such that for all $d,n\geq 1$ and $t>0$ it holds that 
\begin{equation*}
        \Prr{\tilde{S}_{n,k}(\med)\geq t}\lesssim \exp\left(-nr^2M^2+d\log(cn/d)\right)+
     \exp\left(-3MC/4k+d\log 9+\log n\right).
\end{equation*}
\end{lemma}
\begin{proof}[Proof of Lemma~\ref{lem::med_cond_1}]
Observe that
\begin{equation}\label{eqn::median_sen_bound}
\tilde{S}_{n,k}(\mu)\leq  \sup_{u\in\bS^{d-1}}
| (X^\top u)_{(\floor{(n+1)/2}+k)}-(X^\top u)_{(\floor{(n+1)/2})}|. 
\end{equation}
Therefore, applying Lemma~\ref{thm::order_spacing_new_uniform} with $q=1/2$ and $k=k$ gives the result. 
\end{proof}
\begin{lemma}\label{lem::mad_cond_1}
Suppose that Condition~\ref{cond::dens_lb} holds for $(q',\nu,k)$ with $q'<1/4$ and that there exists $q'<q<1/4$ such that $k<n(q-q')/2$. 
Then for all $d,n\geq 1$ and $C>0$ it holds that 
\begin{multline*}
 \Prrr{\tilde{S}_{n,k}(\mad)\geq C/n}\lesssim \exp\left(-nr^2M^2+d\log(cn/d)\right)\\
 +
     \exp\left(-3MC/8k+d\log 9+\log(n(1-2q'+1/n))\right).
\end{multline*}
\end{lemma}
\begin{proof}
Let $$E_{10}=\Big\{\sup_{\substack{u\in\bS^{d-1}\\ p\in [q',1-q'-k/n]}}(X^\top u)_{(\floor{pn}+k)}-(X^\top u)_{(\floor{pn})}< C/2n\Big\}.$$
A direct application of Lemma~\ref{thm::order_spacing_new_uniform} gives that for all $\ubsd$ and $\tilde\nu_u\in\cM(k,\hat\nu)$, 
\begin{multline*}
   \Prr{E_{10}^c}\leq \exp\left(-3MC/8k+d\log 9+\log(n(1-2q'+1/n))\right)\\
+\exp\left(-nr^2M^2+d\log(cn/d)\right).
\end{multline*} 

It remains to show that on $E_{10}$, for all $\ubsd$ and $\tilde\nu_u\in\cM(k,\hat\nu)$, it holds that
$\sigma(\hat\nu_u)-C/n\leq \sigma(\tilde\nu_u)\leq\sigma(\hat\nu_u)+C/n$. 
We begin by proving the upper bound. 
To this end, \eqref{eqn::median_sen_bound} implies that on $E_{10}$ it holds that $\tilde S_{n,k}(\mu)< C/2n.$ 
This fact implies that 
\begin{equation}\label{eqn::subset}
    [\hat\mu_u-\hat\sigma_u,\hat\mu_u+\hat\sigma_u]\subset [\mu(\tilde\nu_u)-\hat\sigma_u-C/2n,\mu(\tilde\nu_u)+\hat\sigma_u+C/2n].
\end{equation}
Next, recall that, by definition, the median absolute deviation is equal to half the length of the smallest interval which is symmetric about the median and which contains $\floor{(n+1)/2}$ of the observations. 
Therefore, $\floor{(n+1)/2}$ of the projected observations are contained in $[\hat\mu_u-\hat\sigma_u,\hat\mu_u+\hat\sigma_u]$. 
Consequentially, \eqref{eqn::subset} implies that on $E_{10}$, for all $\ubsd$ and $\tilde\nu_u\in\cM(k,\hat\nu)$, $\floor{(n+1)/2}$ of the projected observations are also contained in $[\mu(\tilde\nu_u)-\hat\sigma_u-C/2n,\mu(\tilde\nu_u)+\hat\sigma_u+C/2n]$. 
This directly implies that for all $\ubsd$ and $\tilde\nu_u\in\cM(k,\hat\nu)$, it holds that $\sigma(\tilde\nu_u)\leq \sigma(\hat\nu_u)+C/2n$.

It remains to show that on $E_{10}$,
$\sigma(\hat\nu_u)-C/n\leq \sigma(\tilde\nu_u)$. 
First, by definition, for all $\ubsd$ and $\tilde\nu_u\in\cM(k,\hat\nu)$, $\floor{(n+1)/2}-k$ of the projected observations are contained in $[\mu(\tilde\nu_u)-\sigma(\tilde\nu_u),\mu(\tilde\nu_u)+\sigma(\tilde\nu_u)]$. 
Next, note that for all $\ubsd$ there exists $q<p_u< 1-q$ such that $\hat\sigma_u=|\mu(\hat\nu_u)-\hat\xi_{p_u,u}|$. 
To see this, take  $q'<q< 1/4$, and note that by definition,
$$\hat\sigma_u\leq |\hat\xi_{q,u}-\hat\xi_{1/2,u}|\vee |\hat\xi_{1-q,u}-\hat\xi_{1/2,u}|.$$
(This is because $[\hat\xi_{q,u},\hat\xi_{1-q,u}]$ always contains at least half of the observations by definition ($q<1/4$) and, therefore, $\mu_u\pm |\hat\xi_{q,u}-\hat\xi_{1/2,u}|\vee |\hat\xi_{1-q,u}-\hat\xi_{1/2,u}|$ contains at least half of the observations.) 
Therefore, there exists $q<p_u<1-q$ such that $\hat\sigma_u=|\hat\xi_{1/2,u}-\hat\xi_{p_u,u}|$.


Now, by assumption, $2k< n(q-q')$, and so, for all $\ubsd$, there exists $q'<p_u< 1-q'$ such that $\sigma(\tilde\nu_u)=|\mu(\tilde\nu_u)-\hat\xi_{p_u,u}|$. (This is because $\hat\xi_{q,u}$ is at least $2k$ observations away from $\hat\xi_{q',u}$.) Therefore, on $E_{10}$, $\floor{(n+1)/2}$ of the observations lie within
 $[\mu(\hat\nu_u)-\sigma(\tilde\nu_u)-C/n,\mu(\hat\nu_u)+\sigma(\tilde\nu_u)+C/n]$,
which immediately implies that
$\sigma(\hat\nu_u)-C/n\leq \sigma(\tilde\nu_u)$. 
\end{proof}
\begin{proof}[Proof of Theorem~\ref{thm::order_cond_1}]
    The results follows from Lemma~\ref{lem::med_cond_1} and Lemma~\ref{lem::mad_cond_1}.
\end{proof}

\begin{lemma}\label{lem::iqr_cond_1}
Let $\sigma$ be the interquartile range. If Condition~\ref{cond::dens_lb} holds for $(1/4,\nu,k)$, then for all $d,n\geq 1$ and $t>0$ it holds that 
\begin{multline*}
\Prr{\tilde{S}_{n,k}(\sigma)\geq t}\lesssim \exp\left(-3MC/8k+d\log 9+\log(n(1/2+1/n))\right)\\
+\exp\left(-nr^2M^2+d\log(cn/d)\right).
\end{multline*}
\end{lemma}
\begin{proof}[Proof of Lemma~\ref{lem::iqr_cond_1}]
Observe that
\begin{align*}
   \tilde{S}_{n,k}(\sigma)\leq 2\sup_{u\in\bS^{d-1}}
| (X^\top u)_{(\floor{(n+1)/4}+k)}-(X^\top u)_{(\floor{(n+1)/4})}|.
\end{align*}
Therefore, applying Lemma~\ref{thm::order_spacing_new_uniform} with $q=1/4$ and $k=k$ yields the result. 
\end{proof}
\section{Proofs and related lemmas from Section~\ref{sec::ptr_general}}\label{app::gen_proofs}
Let $h(A,\hat\nu)=\int_A\exp(-\phi(x,{\hat{\nu})\epsilon/4\eta})dx$. 
\begin{lemma}\label{lemm::safe_cond}
For any $\hat\nu\in \widehat\cM_1(n,d)$, $\tau,\epsilon>0$, $0<\eta\leq \tau$ and $\delta \leq e^{\epsilon/2}/2$ if
\begin{equation}\label{eqn::cond_1_safety}
     \sup_{\substack{\tilde{\nu}\in \am{\hat\nu}{1}\\ x\in A_{\tau,\hat{\nu}}\cup A_{\tau,\tilde{\nu}}}}|\phi(x,\hat{\nu})-\phi(x,\tilde{\nu})|\leq \eta,
\end{equation}
and 
\begin{equation}\label{eqn::cond_2_safety}
     \frac{h(A_{\tau-\eta,\hat{\nu}},\hat{\nu})} { h(A_{\tau+\eta,\hat{\nu}},\hat{\nu})}\geq 1-e^{-\epsilon/2}\delta/3,
\end{equation}
then $(\EM_{\hat{\nu},\phi}(\cdot;\tau,\eta,\epsilon),\hat{\nu})$ is $(\epsilon/2,\delta)$-safe.
\end{lemma}
\begin{proof}
This proof is inspired by the techniques of \citet{Brown2021}. 
Let $Q=\EM_{\hat{\nu},\phi}(\cdot;\tau,\eta,\epsilon)$. For some $\tilde\nu\in \cM(\hat\nu,1)$, let $P=\EM_{\tilde\nu,\phi}(\cdot;\tau,\eta,\epsilon)$. Take $C=A_{\tau,\tilde{\nu}}\cap A_{\tau,\hat{\nu}}$ and $E\in \mathscr{B}(\rdd)$. Then,
\begin{align*}
    Q(E)&=Q(E\cap C)+Q(E\cap C^c)\\
    &= \frac{Q(E\cap C)}{P(E\cap C)}P(E\cap C)+Q(E\cap A_{\tau,\tilde{\nu}}^c)+Q(E\cap A_{\tau,\hat{\nu}}^c)-Q(E\cap A_{\tau,\hat{\nu}}^c\cap A_{\tau,\tilde{\nu}}^c)\\
    &\leq \frac{Q(E\cap C)}{P(E\cap C)}P(E\cap C)+Q(A_{\tau,\tilde{\nu}}^c)\\
    &\leq \frac{Q(E\cap C)}{P(E\cap C)}P(E)+Q(A_{\tau,\tilde{\nu}}^c).
\end{align*}
Analogously, we have that 
\begin{align*}
   P(E)&\leq \frac{P(E\cap C)}{Q(E\cap C)}Q(E\cap C)+P(A_{\tau,\hat\nu}^c) \leq \frac{P(E\cap C)}{Q(E\cap C)}Q(E)+P(A_{\tau,\hat\nu}^c).
\end{align*}
It remains to show the following bounds: 
\begin{align*}
    Q(A_{\tau,\tilde{\nu}}^c),P(A_{\tau,\hat{\nu}}^c)&\leq  \delta/3, \qquad\text{ and }\qquad
    \frac{Q(E\cap C)}{P(E\cap C)},\frac{P(E\cap C)}{Q(E\cap C)}\leq e^{\epsilon/2}+2\delta/3.
\end{align*}
First, \eqref{eqn::cond_1_safety} implies that for $\tilde\nu\in \cM(\hat\nu,1)$, it holds that both
\begin{equation}\label{eqn::subsets}
    A_{\tau-\eta,\tilde{\nu}}\subset A_{\tau,\hat{\nu}}\subset A_{\tau+\eta,\tilde{\nu}}\qquad\text{and}\qquad A_{\tau-\eta,\hat{\nu}}\subset A_{\tau,\tilde{\nu}}\subset A_{\tau+\eta,\hat{\nu}}.
\end{equation}
In addition, noting that \[h(A,\tilde\nu) = \int_{A} e^{-\phi(x,\tilde\nu)}dx = \int_{A} e^{-(\phi(x,\tilde\nu)-\phi(x,\hat\nu))}e^{-\phi(x,\hat\nu)}dx,\] then \eqref{eqn::cond_1_safety} implies that for any $A\subset A_{\tau,\tilde{\nu}} \cup A_{\tau,\hat{\nu}}$,
\begin{equation}\label{eqn::local_privacy}
e^{-\epsilon/4}h(A,\hat\nu)\leq h(A,\tilde{\nu})\leq e^{\epsilon/4}h(A,\hat\nu). 
\end{equation}
Applying \eqref{eqn::subsets} and \eqref{eqn::local_privacy} together with \eqref{eqn::cond_2_safety} yields
{\small
\begin{align*}
    \frac{Q(E\cap C)}{P(E\cap C)}&\leq e^{\epsilon/4}\frac{h(A_{\tau,\tilde{\nu}},\tilde{\nu}) }{h(A_{\tau,\hat{\nu}},\hat\nu)} \leq e^{\epsilon/2}\frac{h(A_{\tau,\tilde{\nu}},\hat{\nu}) }{h(A_{\tau,\hat{\nu}},\hat\nu)}  \leq e^{\epsilon/2}\frac{h(A_{\tau+\eta,\hat{\nu}},\hat{\nu}) }{h(A_{\tau,\hat{\nu}},\hat\nu)}\leq e^{\epsilon/2}(1-e^{-\epsilon/2}\delta/3)^{-1}\leq e^{\epsilon/2} + 2\delta/3,
\end{align*}}
noting that for $x\leq 1/2$, $(1-x)^{-1}\leq 1+1/x$. 
By an analogous argument, 
\begin{align*}
    \frac{P(E\cap C)}{Q(E\cap C)}&\leq e^{\epsilon/2}+2\delta/3.
\end{align*}

Next, applying \eqref{eqn::subsets} and \eqref{eqn::local_privacy} together with \eqref{eqn::cond_2_safety} yields
\begin{align*}
Q(A_{\tau,\tilde{\nu}}^c)&=1-\frac{h(A_{\tau,\tilde{\nu}}\cap A_{\tau,\hat{\nu}},\hat\nu)}{h(A_{\tau,\hat{\nu}},\hat\nu)}\leq 1-\frac{h(A_{\tau-\eta,\hat{\nu}},\hat\nu)}{h(A_{\tau,\hat{\nu}},\hat\nu)}\leq e^{-\epsilon/2}\delta/3 \leq \delta/3.
\end{align*}
Similarly, we have that 
\begin{align*}
 P(A_{\tau,\hat{\nu}}^c) = \frac{h(A_{\tau,\tilde\nu}\cap A_{\tau,\hat\nu}^c, \tilde\nu)}{h(A_{\tau,\tilde\nu}, \tilde\nu)} &\leq \frac{e^{\epsilon/4}h(A_{\tau,\hat{\nu}}^c\cap A_{\tau,\tilde{\nu}},\hat{\nu})}{e^{-\epsilon/4}h(A_{\tau,\tilde{\nu}},\hat{\nu})}\\
 &\leq e^{\epsilon/2}\left(1-\frac{h(A_{\tau,\hat{\nu}}\cap A_{\tau,\tilde{\nu}},\hat{\nu})}{h(A_{\tau,\tilde{\nu}},\hat{\nu})}\right)\\
 &\leq e^{\epsilon/2}\left(1-\frac{h(A_{\tau-\eta,\hat{\nu}}, \hat\nu)}{h(A_{\tau+\eta,\hat{\nu}},\hat{\nu})}\right)\\
 &\leq e^{\epsilon/2}\left(e^{-\epsilon/2}\delta/3\right) = \delta/3. \qedhere
\end{align*}
\end{proof}
\begin{lemma}\label{thm::safety_margin_lower_bound}
If, for $\hat\nu\in\widehat\cM_1(n,d)$, $k\in[n]$, $\epsilon>0$, $0<\delta\leq e^{-\epsilon/2}/2 $ and $\EM_{\hat\nu,\phi}=\EM_{\hat\nu,\phi}(\cdot;\tau,\eta,\epsilon)$ it holds that
\begin{equation}\label{eqn::cond_1_sm}
  \sup_{\substack{\tilde{\nu}\in \am{\hat\nu}{k}\\ x\in A_{\tau+\eta,\hat{\nu}}\cup A_{\tau+\eta,\tilde{\nu}}}}|\phi(x,\hat{\nu})-\phi(x,\tilde{\nu})|\leq \eta/2
\end{equation}
and 
\begin{equation}\label{eqn::cond_2_sm}
\inf_{y>0}e^{-\epsilon y/4\eta} \frac{\vol(A_{\tau+2\eta,\hat\nu}\cap A_{\tau-2\eta,\hat\nu}^c)}{\vol(A_{\tau-y-2\eta,\hat{\nu}})}\leq e^{-\epsilon/4}\delta/3,
\end{equation}
then $\SM(\hat\nu,\EM_{\hat\nu,\phi},\epsilon/2,\delta)\geq k-1$. 
\end{lemma}
\begin{proof}
To say that $\SM(\hat\nu,\EM_{\hat\nu,\phi},\epsilon/2,\delta)\geq k-1$ is equivalent to saying that, for every $\tilde\nu\in\cM(\hat\nu,k-1)$, $(\EM_{\tilde\nu,\phi},\tilde{\nu})$ is $(\epsilon/2,\delta)$-safe. Using Lemma~\ref{lemm::safe_cond} on $\tilde\nu$, we need only show that \eqref{eqn::cond_1_safety} and \eqref{eqn::cond_2_safety} hold. That is 
\[\sup_{\substack{\nu' \in \am{\tilde\nu}{1}\\ x\in A_{\tau,\tilde{\nu}}\cup A_{\tau,\nu'}}}|\phi(x,\tilde{\nu})-\phi(x,\nu')|\leq \eta,\qquad\text{and}\qquad\frac{h(A_{\tau-\eta,\tilde{\nu}},\tilde{\nu})} { h(A_{\tau+\eta,\tilde{\nu}},\tilde{\nu})}\geq 1-e^{-\epsilon/2}\delta/3.\]

First, the assumption \eqref{eqn::cond_1_sm}, together with the triangle inequality, implies that, for every $v' \in \am{\tilde{\nu}}{1}$ and $x\in A_{\tau+\eta,\tilde{\nu}}\cup A_{\tau+\eta,\nu'}$, \begin{multline*}
    |\phi(x, \nu') - \phi(x, \tilde\nu)| = |\phi(x, \nu') - \phi(x,\hat\nu) + \phi(x,\hat\nu) - \phi(x, \tilde\nu)| \\ 
    \leq 2\sup_{\substack{\tilde{\nu}\in \am{\hat\nu}{k}\\ x\in A_{\tau+\eta,\hat{\nu}}\cup A_{\tau+\eta,\tilde{\nu}}}}|\phi(x,\hat{\nu})-\phi(x,\tilde{\nu})| \leq \eta.
\end{multline*}
This, taken together with the fact that $A_{\tau,\nu}\subseteq A_{\tau+\eta,\nu}$ for every $\nu$, implies that \eqref{eqn::cond_1_safety} holds. 

Next, note that \[\frac{h(A_{\tau-\eta,\tilde{\nu}},\tilde{\nu})} { h(A_{\tau+\eta,\tilde{\nu}},\tilde{\nu})} = \frac{h(A_{\tau-\eta,\tilde{\nu}}\cap A_{\tau+\eta,\tilde{\nu}},\tilde{\nu})}{ h(A_{\tau+\eta,\tilde{\nu}},\tilde{\nu})} = \EM_{\tilde{\nu},\phi}(A_{\tau-\eta,\tilde{\nu}};\tau+\eta,\eta,\epsilon).\] As a result, \eqref{eqn::cond_2_safety} is equivalent to \[\EM_{\tilde{\nu},\phi}(A_{\tau-\eta,\tilde{\nu}};\tau+\eta,\eta,\epsilon) \geq 1 - e^{-\epsilon/2}\delta/3 \implies e^{-\epsilon/2}\delta/3 \geq 1 - \EM_{\tilde{\nu},\phi}(A_{\tau-\eta,\tilde{\nu}};\tau+\eta,\eta,\epsilon).\] Thus, we turn to proving the condition
\begin{equation}\label{eqn::goal}
    1 - \EM_{\tilde{\nu},\phi}(A_{\tau-\eta,\tilde{\nu}};\tau+\eta,\eta,\epsilon) = \frac{h(A_{\tau-\eta,\tilde{\nu}}^c\cap A_{\tau+\eta,\tilde{\nu}},\tilde{\nu})} { h(A_{\tau+\eta,\tilde{\nu}},\tilde{\nu})}\leq e^{-\epsilon/2}\delta/3.
\end{equation}

To this end, note that for any $A\subset A_{\tau+\eta,\tilde{\nu}} \cup A_{\tau+\eta,\hat{\nu}}$, \[e^{-\epsilon/4}h(A,\nu') \leq h(A,\tilde{\nu}) \leq e^{\epsilon/4}h(A,\nu').\] Moreover, \eqref{eqn::cond_1_sm} implies that $A_{\tau,\tilde{\nu}} \subset A_{\tau+\eta, \nu'} \subset A_{\tau+2\eta,\tilde{\nu}}$ and $A_{\tau,\nu'} \subset A_{\tau+\eta, \tilde\nu} \subset A_{\tau+2\eta,\nu'}$. \eqref{eqn::subsets}
Applying these two facts yields that for any $y>0$, it holds that
{\small
\begin{equation}\label{eqn::dnom}
    h(A_{\tau+\eta,\tilde\nu},\tilde\nu)\geq h(A_{\tau-y-2\eta,\tilde\nu},\tilde\nu)\geq e^{-\epsilon/4} h(A_{\tau-y-2\eta,\nu'},\nu') \geq  e^{-\epsilon(\tau-y-\eta)/(4\eta)}\vol(A_{\tau-y-2\eta,\nu'}).
\end{equation}}
Next, we bound the numerator in the left-hand side of \eqref{eqn::goal} using the same two facts,
\begin{equation}\label{eqn::numer}
h(A_{\tau+\eta,\tilde\nu}\cap A_{\tau-\eta,\tilde\nu}^c,\tilde\nu)\leq e^{\epsilon/4}h(A_{\tau+2\eta,\hat\nu}\cap A_{\tau-2\eta,\hat\nu}^c,\hat\nu)\leq e^{-\epsilon(\tau+\eta)/(4\eta)}\vol(A_{\tau+2\eta,\hat\nu}\cap A_{\tau-2\eta,\hat\nu}^c).
\end{equation}
Lastly, combining \eqref{eqn::dnom} and \eqref{eqn::numer} with \eqref{eqn::cond_2_sm} yields
\begin{align*}
   \frac{h(A_{\tau-\eta,\tilde{\nu}}^c\cap A_{\tau+\eta,\tilde{\nu}},\tilde{\nu})} { h(A_{\tau+\eta,\tilde{\nu}},\tilde{\nu})}&\leq e^{-\epsilon(y+\eta)/4\eta} \frac{\vol(A_{\tau+2\eta,\hat\nu}\cap A_{\tau-2\eta,\hat\nu}^c)}{\vol(A_{\tau-y-2\eta,\hat{\nu}})}\leq e^{-\epsilon/2}\delta/3,
\end{align*}
which completes the proof. 
\end{proof}
\begin{proof}[Proof of Lemma~\ref{lem::gen_bound_no_reply}]
By a union bound, for all $k\in \{0\}\cup [n-1]$, it holds that
\begin{align*}
   1-\E{}{\lambda(\hat\nu,\epsilon,\delta)}&=  \Prr{\SM(\hat\nu,Q)+\frac{2}{\epsilon}W<\frac{2\log(1/2\delta)}{\epsilon}}\\
   &\leq \Prr{\SM(\hat\nu,Q)< k\big | \frac{2\log(1/2\delta)}{\epsilon}-\frac{2}{\epsilon}W\leq  k}\\
   &\hspace{12em}+\Prr{\frac{2\log(1/2\delta)}{\epsilon}-\frac{2}{\epsilon}W>k}\\
   &\leq \Prr{\SM(\hat\nu,Q)< k}+\Prr{W<\log(1/2\delta)-\epsilon k/2}.
\end{align*}
It is immediate from the fact that the above bound holds for all $k\in \{0\}\cup [n-1]$ that 
\begin{align*}
   1-\E{}{\lambda(\hat\nu,\epsilon,\delta)}&\leq \inf_{k\in \{0\}\cup [n-1]}\left[\Prr{\SM(\hat\nu,Q)< k}+\Prr{W<\log(1/2\delta)-\epsilon k/2}\right].
\end{align*}
Applying Lemma~\ref{thm::safety_margin_lower_bound} yields that, for all $k\in\{0\}\cup [n-1]$, it holds that
\begin{multline*}
     \Prr{\SM(\hat\nu,Q) < k}\leq \Prrr{    \sup_{\substack{\tilde{\nu}\in \am{\hat\nu}{k+1}\\ x\in A_{\tau+\eta,\hat{\nu}}\cup A_{\tau+\eta,\tilde{\nu}}}}|\phi(x,\hat{\nu})-\phi(x,\tilde{\nu})|>\eta/2}\\
     +\Prr{\inf_{y>0}e^{-\epsilon y/4\eta} \frac{\vol(A_{\tau+2\eta,\hat\nu}\cap A_{\tau-2\eta,\hat\nu}^c)}{\vol(A_{\tau-y-2\eta,\hat{\nu}})}> e^{-\epsilon/4}\delta/3}.
\end{multline*}
Combining the two bounds gives the result. 
\end{proof}
\section{Results concerning Condition~\ref{cond::Pop_Bound}}
Here, we show that the trimmed estimators, median and median absolute deviation satisfy Condition~\ref{cond::Pop_Bound}. 
\begin{lemma}\label{lem::cond_2_med}
Let $\mu$ be the median. 
If Condition~\ref{cond::dens_lb} holds for $(1/2,\nu,0)$ then for all $t\geq 0$ and $n,d\geq 1$
\begin{equation*}
    \Prr{S_{n}(\mu)\geq t}\lesssim \exp\left(-2nM^2t^2+d\log(cn/d)\right).
\end{equation*}
\end{lemma}
\begin{proof}
The proof follows from Lemma~\ref{lem::quantile_process_concen}. 
\end{proof}
\begin{lemma}\label{lem::cond_2_mad}
Let $\sigma$ be the median absolute deviation. 
If for some $q'$ that satisfies $$0<q'<\inf_{\ubsd}\argmin_{0<p\leq 1/2}|\xi_{1/2,u}-\xi_{p,u}-(\xi_{p+1/2,u}-\xi_{1/2,u})|,$$ Condition~\ref{cond::dens_lb} holds for $(q',\nu,0)$ with $r\geq 1/3M$,
and $\sup_{\substack{x\in\mathcal{I}; \ubsd}}f_{\nu_u}(x)<L<\infty$, then for all $t\geq 0$ and $n,d\geq 1$
\begin{equation*}
     \Prr{S_{n}(\sigma)\geq t}\lesssim \exp\left(-2nM^2t^2/16^2\vee -2n[(1/2-q')M/L\wedge 1/3]^2+d\log(cn/d)\right).
\end{equation*}
\end{lemma}
\begin{proof}
Let $$\cC_u(p)=\xi_{1/2,u}-\xi_{p,u}-(\xi_{p+1/2,u}-\xi_{1/2,u}),$$
and $q_u=\argmin_{0<p\leq 1/2}|\cC_u(p)|$. 
Now, note that under absolute continuity of $\nu_u$, it holds that $\sigma(\nu_u)=|\xi_{1/2,u}-\xi_{q_u,u}|$. 
This is because the median absolute deviation of a measure $P$ is half the length of the symmetric interval about the median, which contains 1/2 of the probability mass of $P$ \citep{Rousseeuw1993}. 
Now, analogously, let
$$\hat\cC_u(p)=\hat\xi_{1/2,u}-\hat\xi_{p,u}-(\hat\xi_{p+1/2,u}-\hat\xi_{1/2,u})\qquad\text{and}\qquad\tilde\cC_u(p)=\hat\xi_{1/2,u}-\hat\xi_{p-1/2,u}-(\hat\xi_{p,u}-\hat\xi_{1/2,u}),$$
with 
$$\hat q_u=\argmin_{0\leq p\leq 1/2}|\hat\cC_u(p)|\qquad\text{and}\qquad \tilde q_u=\argmin_{1/2\leq p\leq 1}|\tilde\cC_u(p)|.$$
It follows by definition of the median absolute deviation that for any $\ubsd$, 
$$|\hat\xi_{1/2,u}-\hat\xi_{\hat q_u,u}|\wedge |\hat\xi_{1/2,u}-\hat\xi_{\tilde q_u,u}|\leq \sigma(\hat\nu_u)\leq  |\hat\xi_{1/2,u}-\hat\xi_{\hat q_u,u}|\vee |\hat\xi_{1/2,u}-\hat\xi_{\tilde q_u,u}|.$$ 
This fact directly implies that for any $\ubsd$, either
{\small
\begin{equation}\label{eqn::sigma_bound_1}
    |\sigma(\hat\nu_u)-\sigma(\nu_u)|\leq|\xi_{1/2,u}-\xi_{q_u,u}-(\hat\xi_{1/2,u}-\hat\xi_{\hat q_u,u})|\leq |\xi_{1/2,u}-\hat\xi_{1/2,u}|+|\xi_{\hat q_u,u}-\xi_{q_u,u}|+|\xi_{\hat q_u,u}-\hat\xi_{\hat q_u,u}|,
\end{equation}}
holds, or 
\begin{equation}\label{eqn::sigma_bound_2}
    |\sigma(\hat\nu_u)-\sigma(\nu_u)|\leq |\xi_{1/2,u}-\hat\xi_{1/2,u}|+|\xi_{\tilde q_u,u}-\xi_{q_u+1/2,u}|+|\xi_{\tilde q_u,u}-\hat\xi_{\tilde q_u,u}|,
\end{equation}
holds. 
Next, we bound $\hat q_u$ below and $\tilde q_u$ above for all $\ubsd$. 
Consider $w\in\re$ such that $q'<w<\inf_{\ubsd} q_u$. 
Now, it follows from Condition~\ref{cond::dens_lb} and an application of the mean value theorem that $$\xi_{w,u}-\xi_{q',u}\leq (w-q')/\inf_{\substack{ \ubsd\\ x\in(\xi_{q',u},\xi_{w,u})}}f_u(x)\leq M^{-1}.$$
Now, using Condition~\ref{cond::dens_lb} and the assumption that $r>1/3M$, together with an application of Lemma~\ref{lem::quantile_process_concen} yields
\begin{equation}\label{eqn::quantiles_close_w}
   \Prrr{\sup_{\substack{ \ubsd\\x\in \{w,1/2\}}}|\hat\xi_{x,u}-\xi_{x,u}|>1/3M}\lesssim  \left(\frac{cn}{d}\right)^de^{-2n/9}. 
\end{equation}
Next, we show that the probability that $(\hat\xi_{w,u},\hat\xi_{1/2,u}+(\hat\xi_{1/2,u}-\hat\xi_{w,u}))=(\hat\xi_{w,u},2\hat\xi_{1/2,u}-\hat\xi_{w,u})$ contains less than 1/2 of the projected observations for at least one $\ubsd$ is small. 
This immediately implies that $\hat q_u\geq w>q'$ for any $\ubsd$. 
To this end, \eqref{eqn::quantiles_close_w} implies that with high probability, $$(\hat\xi_{w,u},2\hat\xi_{1/2,u}-\hat\xi_{w,u})\supset (\xi_{w}+1/3M,2\xi_{1/2,u}-\xi_{w}-1/M)\supset (\xi_{q'},2\xi_{1/2}-\xi_{q'}).$$ 
Now, we bound the probability that $(\xi_{q',u},2\xi_{1/2,u}-\xi_{q',u})$ contains less than 1/2 of the projected observations. 
This is equivalent to bounding:
\begin{equation}\label{eqn::two_terms}
        \Prrr{\bigcup_{\ubsd}\{ \hat\xi_{1/2,u}<\xi_{q',u}\}}\qquad\text{and}\qquad\Prrr{\bigcup_{\ubsd} \{\hat\xi_{1/2,u}>2\xi_{1/2,u}-\xi_{q',u}\}}.
\end{equation}
Starting with the left-hand term, utilizing the mean value theorem, together with the upper bound on the projected density and an application of Lemma~\ref{lem::quantile_process_concen}, yields
\begin{align*}
   \Prrr{\bigcup_{\ubsd}\{ \hat\xi_{1/2,u}<\xi_{q',u}\}}&\leq  \Prr{\sup_{\ubsd} |\hat\xi_{1/2,u}-\xi_{1/2,u}|>\inf_{\ubsd}[\xi_{1/2,u}-\xi_{q',u}]}\\
    &\leq \Prr{\sup_{\ubsd} |\hat\xi_{1/2,u}-\xi_{1/2,u}|>(1/2-q')/L}\\
    &\leq\left(\frac{cn}{d}\right)^de^{-2nM^2(1/2-q')^2/L^2}.
\end{align*}
Similarly, the right-hand term in \eqref{eqn::two_terms} satisfies
\begin{align*}
   \Prrr{\bigcup_{\ubsd} \{\hat\xi_{1/2,u}>\xi_{1/2,u}+\xi_{q',u}\}}&\leq\left(\frac{cn}{d}\right)^de^{-2nM^2(1/2-q')^2/L^2}.
\end{align*}
Thus, $$\Prr{\inf_{\ubsd}\hat q_u<q'}\lesssim \left(\frac{cn}{d}\right)^de^{-2nM^2[(1/2-q')/L\wedge M/3]^2}.$$
An analogous argument gives that  
$$\Prrr{\sup_{\ubsd}\tilde q_u>1-q'}\lesssim  \left(\frac{cn}{d}\right)^de^{-2nM^2[(1/2-q')/L\wedge M/3]^2}.$$
Let $E_{11}=\{q'<\hat q_u<\tilde q_u\leq 1-q'\}$ and note that we have just shown that
\begin{equation}\label{eqn::q_bounds}
    \Prr{E_{11}^c}\lesssim \left(\frac{cn}{d}\right)^de^{-2nM^2[(1/2-q')/L\wedge M/3]^2}.
\end{equation}

The next step is to show that $\{\sup_{\ubsd}|\xi_{\hat q,u}-\xi_{q,u}|\geq t|E_{11}\}\subset \{\sup_{\substack{\ubsd\\ q'<p<1-q'}}|\xi_{p,u}-\hat \xi_{p,u}|\gtrsim t|E_{11}\}$. 
To this end, observe that 
\begin{align*}
   \{\sup_{\ubsd}|\xi_{\hat q,u}-\xi_{q,u}|\geq t|E_{11}\}&=\bigcup_{\ubsd}\{\inf_{\substack{q'<p<1/2\\ |\xi_{p,u}-\xi_{q,u}|>t}} |\hat\cC_u(p)|-|\hat\cC_u(q)|\leq 0\}\\
   &=\bigcup_{\ubsd}\{|\hat\cC_{u}(q)|-\inf_{\substack{q'<p<1/2\\ |\xi_{p,u}-\xi_{q,u}|>t}} |\hat\cC_{u}(p)|\geq 0\}\\
   &\subset\{8\sup_{\substack{\ubsd\\ q'<p<1-q'}}|\xi_{p,u}-\hat\xi_{p,u}|\geq \inf_{p<1/2,\ |\xi_{p,u}-\xi_{q,u}|>t} |\cC_u(p)|\}.
\end{align*}
It remains to show that for any $\ubsd$, $\inf_{0<p<1/2,\ |\xi_{p,u}-\xi_{q,u}|>t} |\cC_u(p)|\gtrsim t$. 
This is equivalent to
\begin{align*}
 \inf_{\substack{0<p<1/2\\ |\xi_{p,u}-\xi_{q,u}|>t}} |\xi_{q,u}-\xi_{p,u}+\xi_{q+1/2,u}-\xi_{p+1/2,u}|\gtrsim t.
\end{align*}
Now,
\begin{multline*}
      \inf_{\substack{0<p<1/2\\ |\xi_{p,u}-\xi_{q,u}|>t}}|\xi_{q,u} - \xi_{p,u} + \xi_{q+1/2,u} - \xi_{p+1/2,u}|\\
      =\inf_{\substack{0<p<1/2\\ |\xi_{p,u}-\xi_{q,u}|>t}}[ |\xi_{q,u} - \xi_{p,u}| + |\xi_{q+1/2,u} - \xi_{p+1/2,u}|]\geq t.
\end{multline*}
The first equality comes from the fact that $|\xi_{p,u}-\xi_{q,u} + \xi_{q+1/2,u} - \xi_{p+1/2,u}|$ is of the form of $|a + b|$ where $sign(a) = sign(b)$. 
The second inequality comes from the fact that we are minimizing over $|\xi_{p,u}-\xi_{q,u}|>t$ and the fact that the second term is non-negative. 
Thus,
\begin{equation}\label{eqn::hard_step}
       \{|\xi_{\hat q,u}-\xi_{q,u}|\geq t|E_{11}\}\subset\{\sup_{q'<p<1-q'}|\xi_{p,u}-\hat\xi_{p,u}|\geq t/8\}.
\end{equation}
By an analogous argument,
\begin{equation}\label{eqn::hard_step_2}
       \{|\xi_{ q,u}-\xi_{\tilde q,u}|\geq t|E_{11}\}\subset\{\sup_{q'<p<1-q'}|\xi_{p,u}-\hat\xi_{p,u}|\geq t/8\}.
\end{equation}
Now, combining equations \eqref{eqn::sigma_bound_1}, \eqref{eqn::sigma_bound_2}, \eqref{eqn::q_bounds}, \eqref{eqn::hard_step} and \eqref{eqn::hard_step_2},
\begin{align*}
    \Prr{S_{n}(\sigma)\geq t}&\lesssim \Prr{\sup_{q'<p<1-q'}|\xi_{p,u}-\hat\xi_{p,u}|\geq t/16}+\left(\frac{cn}{d}\right)^de^{-2n[(1/2-q')M/L\wedge 1/3]^2}\\
    &\lesssim 
    \left(\frac{cn}{d}\right)^de^{-2nM^2t^2/16^2}+\left(\frac{cn}{d}\right)^de^{-2n[(1/2-q')M/L\wedge 1/3]^2}. \qedhere
\end{align*}
\end{proof}
\begin{lemma}\label{lemm::cond2-gaussian-trimmed}
Suppose that $\nu=\cN(\mu,\Sigma)$. Then there exists a universal constant $c>0$ such that for all $0\leq \alpha<1/2$ Condition~\ref{cond::Pop_Bound} holds for $(\nu,\tm_\alpha,\tad_\alpha)$ with $$\zeta_{n,\mu}(t)=\zeta_{n,\sigma}(t)=e^{-cn(1-2\alpha)t^2/\norm{\Sigma^{-1/2}}^2+d}.$$
\end{lemma}
\begin{proof}
WLOG, assume that $\mu=0$, and let $Z_\alpha\sim N(0,\Sigma/n(1-2\alpha))$. 
Then, it follows from Gaussian concentration of measure
\begin{equation}\label{eqn::tmpb}
       \Prr{S_{n}(\tm_\alpha)\geq t}= \Prr{\sup_u |\bar X_{\alpha}^\top u|\geq t}=\Prr{\norm{\bar X_{\alpha}}\geq t}    \lesssim e^{-n(1-2\alpha)t^2/\norm{\Sigma^{-1/2}}^2}.
\end{equation}
Now, we focus on $\Prr{S_{n}(\tad_\alpha)\geq t}$. 
First, we have that $$\Prr{S_{n}(\tad_\alpha)\geq t}\leq \Prr{S_{n}(\tad_\alpha)\geq t/\norm{\Sigma^{-1/2}}}.$$ 
From now on, assume that $\nu=\cN(0,I)$ and let $\sam{Z}{n}\sim N(0,I)$. 
\begin{align*}
   S_{n}(\tad_\alpha)&=\sup_{\ubsd} |\tad_\alpha(\hat\nu_u)-\tad_\alpha(\nu_u)|\\
   &= \sup_{\ubsd} |\frac{1}{n(1-2\alpha)}\sum_{i=\alpha n}^{n-\alpha n} |X_i^\top u-\bar X_{\alpha}^\top u|-\int|Z^\top u|d\nu|\\
   &= \sup_{\ubsd} \frac{1}{n(1-2\alpha)}\sum_{i=\alpha n}^{n-\alpha n}| |X_i^\top u-\bar X_{\alpha}^\top u|-\sqrt{2\pi}|\\
   &= \sup_{\ubsd} \frac{1}{n(1-2\alpha)}\sum_{i=\alpha n}^{n-\alpha n}||Z_i^\top u| -\sqrt{2\pi}|+\sup_{\ubsd}|\bar X_{\alpha}^\top u-\mu_u|\\
   &= \frac{1}{n(1-2\alpha)}\sup_{\ubsd} \norm{|\sam{Z}{n}^\top u| -\sqrt{2\pi}\mathbbm{1}}_1+\sup_{\ubsd}|\bar X_{\alpha}^\top u-\mu_u|\\
   &\leq \frac{1}{\sqrt{n(1-2\alpha)}}\sup_{\ubsd} \norm{|\sam{Z}{n}^\top u| -\sqrt{2\pi}\mathbbm{1}}_2+\sup_{\ubsd}|\bar X_{\alpha}^\top u-\mu_u|\\
   &\coloneqq I+II.
\end{align*}
Applying this bound, along with a union bound yields that there exists a universal constant $C>0$ such that 
\begin{align*}
    \Prr{S_{n}(\tad_\alpha)\geq t}&\leq \Prr{I+II\geq Ct}\leq \Prr{I\geq Ct/2}+\Prr{II\geq Ct/2}.
\end{align*}
Now, for $II$, applying \eqref{eqn::tmpb} yields that 
\begin{equation}
    \Prr{II\geq Ct/2}\leq e^{-n(1-2\alpha)t^2/\norm{\Sigma^{-1/2}}^2}. 
\end{equation}

We now focus on $I$, and apply an $\epsilon$-net argument. 
Recall that the set $\cN_\epsilon$ is an $\epsilon$-net on $\bS^{d-1}$ if for all $\ubsd$, there exists $v\in \cN_\epsilon$ such that $\norm{u-v}\leq \epsilon$. 
Letting $\cN_\epsilon$ be an arbitrary $\epsilon$-net on $\bS^{d-1}$, it is straightforward (see, e.g., \citep[][Exercise 4.4.4]{vershynin2018high}) to show that for any $c>0$, $y\in \rdd$ and $0<\epsilon<1/2$,
$$\sup_{\ubsd}||y^\top u|-c|\leq \frac{C}{1-2\epsilon}\sup_{u\in \cN_\epsilon}||y^\top u|-c|.$$
Furthermore, there exists $\cN_{1/4}$ such that $|\cN_{1/4}|\leq 9^d$ \citep[][see Corollary 4.2.13]{vershynin2018high}. 
Applying these facts, along with Gaussian concentration of measure yields that there exists a universal constant $c>0$ such that
\begin{align*}
    \Prr{I\geq Ct/2}&\leq 9^d\sup_{u\in\cN_{1/4}}\Prr{\frac{1}{\sqrt{n(1-2\alpha)}} \norm{|\sam{Z}{n}^\top u| -\sqrt{2\pi}\mathbbm{1}}_2}\\
    &=9^d\Prr{\frac{1}{\sqrt{n(1-2\alpha)}} \norm{|\sam{Z}{n}'| -\sqrt{2\pi}\mathbbm{1}}_2}\\
    &\lesssim e^{-n(1-2\alpha)ct^2+d}.
\end{align*}
This completes the proof. 

 
\end{proof}
\section{Cauchy lower bound}
Here, we present a lower bound on the sample complexity of estimating the location parameter from a measure with Cauchy marginals.  
For a set $S$, let $\tilde\cM_{1,\epsilon,n}(S)$ be the set of maps from $\hat\nu$ to $\cM_1(S)$ which satisfy \eqref{eqn::dp}. 
Note that $\tilde\cM_{1,\epsilon,n}(S)$ is essentially the set of all differentially private estimators which lie in $S$. 
Let $\cC$ be the set of measure with Cauchy marginals over $\rdd$. 
\begin{theorem}\label{lem::s_c_cauchy}
For all $n,d\geq 1$, $t\geq 0$ it holds that 
$$n=\Omega\left(\frac{d}{t ^2}\vee \frac{d}{t\epsilon}\right),$$
samples are required for 
\begin{equation}\label{eqn::mmlb}
    \inf_{T_\epsilon\sim \tilde\cM_{1,\epsilon,n}(\re)}\sup_{\nu\in \cC}\E{}|T_\epsilon(\hat\nu)-\theta(\nu)| \leq t . 
\end{equation}
\end{theorem}
\begin{proof}
We apply Corollary 4 of \citep{Acharya2021}, of which a simpler version is restated below for clarity. 
For $\cP\subset \cM_1(\re)$, let $n^*=\inf\{n\colon \inf_{T_\epsilon\sim \tilde\cM_{1,\epsilon,n}(\re)}\sup_{\nu\in \cP}\int |T_\epsilon(\hat\nu)-T_\epsilon(\nu) |d\nu^n\leq \tau\}.$
\begin{corollary}[\cite{Acharya2021}]\label{lemm::p_assouad}
For all $\epsilon,\tau>0$ and any $\cP\subset \cM_1(\re)$, let $\cQ=\{Q_1,\ldots,Q_m\}\subset \cP$. 
If for all $i\neq j$, it holds that $|T_\epsilon(Q_i)-T_\epsilon(Q_j) |\geq 3\tau$, $\KL(Q_i,Q_j)\leq \beta$, and $\TV(Q_i,Q_j)\leq \gamma$, then
then $n^*= \Omega\left(\log m(\beta^{-1}\vee (\gamma\epsilon)^{-1})\right)$. 
\end{corollary}
Let $P_{\theta}$ be  the Cauchy measure with location parameter $\theta\in \re$ and let $P_{\theta}=\prod _{i=1}^n P_{\theta_i}$ for $\theta=(\theta_1,\ldots,\theta_d)\in\rdd$. 
Now, consider a 1/2-packing of the unit ball in $\rdd$: $V=\{v_1,\ldots,v_M\}$. 
Recall that the packing number of the unit ball is at least $2^d$, and so $M\geq 2^d$. 
Next, define $\Theta=\{\theta_1,\ldots,\theta_M\}$, where $\theta_i=4\alpha v_i$ for some $\alpha>0$. 
It follows that for $i\neq j$,  $\norm{\theta_i-\theta_j}=4\alpha\norm{v_i-v_j}\geq \alpha$ and that $\norm{\theta_i-\theta_j}\leq \norm{\theta_i}+\norm{\theta_j}\leq 8\alpha$. 

Now, let $\KL(P,Q)$ denote the Kullback–Leibler divergence between two measures $P$ and $Q$. 
We now bound the Kullback–Leibler divergence between $\KL(P_{\theta_i},P_{\theta_j})$ for $j\neq i$. 
To this end, we have that 
\begin{align*}
\KL(P_{\theta_i},P_{\theta_j})&=\sum_{k=1}^d\log \left(1+|\theta_{ik}-\theta_{jk}|^2/4\right)\leq  \norm{\theta_{i}-\theta_{j}}^2/4\leq 16\alpha^2.
\end{align*}
Applying Corollary~\ref{lemm::p_assouad}, we have that 
$n^*=\Omega(\frac{d}{\delta^2}\vee \frac{d}{\delta\epsilon})$.
\end{proof}
\end{document}